\documentclass[12pt]{amsart}

\usepackage{mystyle}

\begin{document}

\title{
From the free boundary condition for Hele-Shaw to a fractional parabolic equation
}

\author[H. A. Chang-Lara]{H\'ector A. Chang-Lara}
\address{Department of Mathematics, Columbia University, New York, NY 10027}
\email{changlara@math.columbia.edu}

\author[N. Guillen]{Nestor Guillen}
\address{Department of Mathematics, University of Massachusetts Amherst, Amherst, MA 01003}
\email{nguillen@math.umass.edu}

\newcommand{\BlueComment}[1]{\color{blue}{#1}\color{black}}
\newcommand{\RedComment}[1]{\color{red}{#1}\color{black}}

\begin{abstract}
We propose a method to determine the smoothness of sufficiently flat solutions of one phase Hele-Shaw problems. The novelty is the observation that under a flatness assumption the free boundary --represented by the hodograph transform of the solution- solves a nonlinear integro-differential equation. This nonlinear equation is linearized to a (nonlocal) parabolic equation with bounded measurable coefficients, for which regularity estimates are available.  This fact is used to prove a regularity result for the free boundary of a weak solution near points where the solution looks sufficiently flat. More concretely, flat means that in a parabolic neighborhood of the point the solution lies between the solutions corresponding to two parallel flat fronts a small distance apart --a condition that only depends on the the local behavior of the solution. In a neighborhood of such a point, the free boundary is given by the graph of a function whose spatial gradient enjoys a universal H\"older estimate in both space and time.
\end{abstract}

\subjclass{35R35,35B65,35R09}
\keywords{Hele-Shaw, free boundary problems, $C^{1,\alpha}$ estimates, nonlocal equations}

\maketitle

\markboth{H. Chang-Lara and N. Guillen}{From Hele-Shaw to a Fractional Parabolic Equation}


\section{Introduction}\label{sec:intro}

The Hele-Shaw model can be used to describe an incompressible flow lying between two nearby horizontal plates \cite{saffman1958penetration}. By renormalizing constants and assuming negligible effects from the vertical components of the velocity and surface tension, it can be understood in terms of a pressure function $u:\W\times(-1,0]\to[0,\8)$ that satisfies,
\begin{alignat}{3}
\label{eq:harmonic} \D u &= 0 \quad &&\text{ in } \quad &&\W_u^+:=\{u>0\},\\
\label{eq:free_boundary} \frac{\p_t u}{|Du|} &= |Du| \quad &&\text{ on } \quad  &&\G_u:=\p \supp u\cap \W.
\end{alignat}
The first equation expresses the incompressibility of the fluid which occupies the domain $\W_u^+$ which is spreading over time. By implicit differentiation, $\p_t u/|Du|$ at the \emph{free boundary} $\G_u$ corresponds to the (outer) normal speed of the interphase between the are occupied by fluid $\W_u^+$ and the empty region. The second relation then indicates that this interphase advances with the speed of the fluid.

The aim of this paper is to illustrate the relationship between the free boundary in the Hele-Shaw problem and solutions to parabolic integro-differential equations, and to show how this can be exploited to analyze the free boundary. Heuristically speaking we observe that, if the free boundary is described by the graph of a function $\bar u:\mathbb{R}^{n-1}\times I \to\mathbb{R}$, and if the solution itself is $\varepsilon$-flat for some $\varepsilon>0$ (i.e. close to a planar front), then $\bar u$ solves the equation
\begin{align*} 
  & \partial_t \bar u = a\left ( \frac{L\bar u}{1-\varepsilon L\bar u} + \varepsilon \frac{|D\bar u|^2}{1-\varepsilon L\bar u} \right ),
\end{align*}
where $a\in [\lambda,\Lambda]$, and $L\bar u$ has the form $-(-\Delta)^{1/2}\bar u + \textnormal{\emph{error}}$, where the \emph{error} term is some quantity going to zero as $\varepsilon \to 0$. This suggests that blow ups of the free boundary (at least near flat points) are governed by the fractional heat equation
\begin{align}
  \partial_t \bar u +a(-\Delta)^{1/2}\bar u = 0. \label{eqn:intro frac heat equation}
\end{align}
Given that global solutions to \eqref{eqn:intro frac heat equation} are differentiable in space and time, one is tempted to seek a proof of H\"older continuous differentiability of the free boundary through compactness and blow up arguments, as has been done in multiple contexts such as the theory of phase transitions, degenerate elliptic PDE, and free boundary problems.

The main result is bellow. See the discussion at the beginning of Section \ref{sec:prelim} for a complete review of our notation. Moreover, the notion of viscosity solution, along with the existence and uniqueness theory developed by Kim \cite{kim2003uniqueness}, is reviewed in Section \ref{sec:viscosity_sol}. 
\begin{theorem}\label{thm: main}
Let $u:B_1\times(-1,0]\to[0,\8)$ be a viscosity solution of,
\begin{alignat*}{3}
\D u &= 0 \quad &&\text{ in } \quad &&\W_u^+ = \{u>0\},\\
\frac{\p_t u}{|Du|} &= |Du| \quad &&\text{ on } \quad  &&\G_u = \p\supp u\cap B_1.
\end{alignat*}
There exists universal constants $\e_0,\a\in(0,1)$ and $C>0$ such that if for some $\e\in(0,\e_0)$,
\begin{align}
(x_n+t-\e)_+ \leq u(x,t) \leq (x_n+t+\e)_+, \label{eqn: main thm flatness}
\end{align}
then for every $t\in(-1/2,0]$, the free boundary can be parametrized as a $C^{1,\a}$ graph in the $e_n$ direction,
\[
\G_u(t) = \{(x',x_n)\in B_{1/2}: x_n=-t-\e\bar u(x',t)\},
\]
with the estimate,
\begin{align*}
\|D_{\R^{n-1}}\bar u\|_{C^\a\1B_{1/2}^{n-1}\times(-1/2,0]\2} \leq C.
\end{align*}
\end{theorem}

Let us make some initial remarks on the the nature of the proof Theorem \ref{thm: main} and highlight some issues that required considerable attention. Firstly, the ``linearization'' of the free boundary condition hinted at above does not yield exactly \eqref{eqn:intro frac heat equation}, but an integro-differential parabolic equation with bounded measurable coefficients, specifically, there will be a time-dependent coefficient in front of $(-\Delta)^{1/2}$ in the linearized equation. This coefficient will not be necessarily continuous in time, but it will be still be bounded between two positives constants, so we still have at our disposal a H\"older regularity estimate. All of this indicates that the theory of integro-differential parabolic equations deserves consideration along the array of tools used in the analysis of Hele-Shaw type flows. It would be worthwhile to investigate to what extent this tools yield answers to unresolved issues such as regularity for Hele-Shaw problems in heterogeneous media, problems without variational structure, and two-phase problems.

In hindsight, the irregularity of the linearized equation has to do with the potentially high oscillation in time of the slope of $u$ near $\Gamma_u$, which will be reflected in the compactness and subsequent ``linearization'' argument. Recall there are global solutions of the problem with low regularity with respect to time: given any continuous, strictly positive function $a(t)$, we have the following spatially flat, global solution to Hele-Shaw
\begin{align}
  v(x,t) := a(t)\max\1(x_n-A(t)),0\2, \qquad A(t) = \int_t^0 a(s)ds.\label{eqn:intro flat solutions}
\end{align}
The free boundary regularity in time is limited by the regularity of $a(t)$, and to obtain further regularity with respect to time one must impose further assumptions. We contend that this issue is closely related to the differentiability in time to solutions of parabolic integro-differential equations when the Dirichlet data is not regular in time. Indeed, the nonlocal effects mean that lack of smoothness of the Dirichlet data with respect to time may affect the interior differentiability of the solution (see \cite[Section 6]{MR3148110} for an example). Note also that, in any case, the spatial normal to $\Gamma_v$ is constant in space and time (always equal to $e_n$), which corresponds to $D_{\mathbb{R}^{n-1}}\bar v = 0$ in the terminology of Theorem \ref{thm: main}, showing the result holds trivially for these spatially flat fronts. 

Of course, another well known scenario giving rise to irregular behavior in time for $u$ is whenever two pieces of the free boundary approach each other leading to a topological change.

\begin{remark} Later in Corollary \ref{cor:bdd_measurable} we state a more general result for the case where the solution is close to a planar profile with variable slopes that change continuously in time --the resulting estimate being independent of the modulus of continuity.

\end{remark}

For a discussion of results related to Theorem \ref{thm: main}, including the regularity theory developed in works of Choi, Kim, and Jerison (crucially, \cite{MR2306045,MR2603767}) see Section \ref{sec:related results} below. For now, let us  highlight some overall differences between Theorem \ref{thm: main} and the results in \cite{MR2306045,MR2603767}. The latter prove $C^1$ differentiability both in space and time of the free boundary under an extra assumption on the regularity of $u$ with respect to time (see \cite[Theorem 1.2]{MR2603767}). It is clear an assumption of this kind must be imposed, or else there is no way of preventing the spatially flat fronts with an an arbitrary $a(t)$ that were discussed above. However, the methods in \cite{MR2306045,MR2603767} do not seem to use this assumption in proving the $C^1$ regularity of $\Gamma_u$ in space for a fixed time. On the other hand, Theorem \ref{thm: main} proves $C^\alpha$ continuity not only in space but also in time for the spatial normal to $\Gamma_u$. This is stronger than just $C^1$ regularity in space of $\Gamma_u$ for fixed time, but does not go as far as $C^1$ regularity in space and time for $\Gamma_u$ (which as illustrated above, requires further assumptions). This goes back to the point made above regarding the relation between the boundary data and the interior differentiability of solutions to nonlocal parabolic equations. In particular, it can be said Theorem \ref{thm: main} and the results in \cite{MR2603767}  although having some overlap in cases they treat, they ultimately deal with different aspects of the problem and neither result is contained in the other -and each employs different methods. 

Furthermore, it is worth to recall that the evolution is an eminently nonlocal flow from the perspective of the free boundary $\Gamma_u$ itself,  while the hypothesis of the theorem entail just a local condition by looking not just at $\Gamma_u$ but the solution $u$ itself \eqref{eqn: main thm flatness}. As such the result holds regardless of far away behavior of $\Gamma_u$, except that which may prevent the validity of \eqref{eqn: main thm flatness}, of course. In particular, since it only concerns the behavior of the function in some space-time cylinder, Theorem \ref{thm: main} applies as an interior result to solutions to Hele-Shaw problems on general domains, regardless of the conditions imposed on $u$ along the fixed boundary or at infinity.

\subsection{Literature overview}\label{sec:related results}

The Hele-Shaw flow is one of the simplest models of interphase evolution, arising in fluid mechanics \cite{saffman1958penetration,richardson1972hele} and appearing in many guises throughout mathematics. It's relation to the porous flow has For instance, work of Caginalp \cite{Caginalp1989} and later Caginalp and Chen \cite{Caginalp1998ConvergencePhaseField} shows how the Hele-Shaw and Stefan problems arise as sharp interface limits of phase field models --in which case one can also obtain more accurate system involving surface tension effects. The Hele-Shaw problem also appears as a limit for the porous medium equation when the power in the nonlinearity goes to infinity, see work of Elliot et al \cite{elliott1986mesa}. Finally, we mention the connection of Hele-Shaw and Stefan type problems with models for internal diffusion-limited aggregation (internal DLA):  Gravner and Quastel showed in \cite{gravner2000internal} that the hydrodynamic limit the density for particles in internal DLA converges to a solution of a one phase Stefan problem.

\textbf{Existence and uniqueness.} There is a wide literature regarding the existence, uniqueness and regularity of solutions. Short time existence of a classical solution starting from smooth initial conditions was done by Escher and Simonett in \cite{escher1997classical}. A variational approach was set forth by Elliott and Janovsk{\`y} in \cite{elliott1981variational}, formulating the problem in terms of the time integral of $u$, which is shown to solve a variational inequality. Existence and uniqueness for viscosity solutions, including a comparison principle, was proved by Kim in \cite{kim2003uniqueness}. As part work of subsequent work by Kim and Mellet dealing with homogenization, they determine in \cite{kim2009homogenization} conditions under which the variational and viscosity formulations coincide.

\textbf{Regularity results: Comparison arguments.} The existence of global in time smooth solutions was obtained by Daskalopoulos and Lee \cite{DaskalopoulosLee2004AllTimeSmoothSolutions} under a smoothness and convexity assumption on the initial condition. The first regularity results for flat interfaces of viscosity solutions can be found in works of Kim \cite{MR2218549,kim2006regularity}. Subsequently, Jerison and Kim studied (in the planar case) the evolution of the problem starting from singular initial data \cite{MR2203166}, determining the exact asymptotic behavior of the free boundary at a singular point. Such analysis of the asymptotic behavior was later done in higher dimensions by Choi, Jerison and Kim \cite{MR2306045}, along with a Lipschitz/flatness implies differentiability result for a problem with constant Dirichlet data \cite{MR2603767}. Shortly after this was generalized and improved in the follow up work \cite{MR2203166}, in particular they show that the free boundary (starting from an initial Lipschitz interface) improves its flatness in a manner which is proportional to its displacement (at least for small times): if a point of the free boundary has moved an amount $\delta$ away from its initial configuration (which is assumed Lipschitz), then near that displaced point the free boundary  will have a flatness of order $\delta$, the constant depending on the initial Lipschitz condition on the free boundary. Their results for instance say that an initial data given by a global Lipschitz graph will become smooth and remain smooth for all later times. Furthermore, their results imply a quantitative version of Sakai's theorem for variational solutions in two dimensions \cite{sakai1993regularity}. These results are deeply connected, and build upon Caffarelli's theory on the elliptic free boundary problems \cite{caffarelli1989harnack}, and the theory of Athanasopoulos, Caffarelli, and Salsa for the two-phase Stefan problem \cite{MR1397563} (see also the discussion further below).

There are a several important themes involved in the proofs in \cite{MR2306045,MR2603767} that we shall review superficially. On one hand, there is the use the interior Harnack inequality and barrier arguments to propagate the initial Lipschitz property (or flatness) of the free boundary for a small positive time. At the same time, it is important to determine how $u$ grows away from the free boundary for a small time interval -and here the tools used to analyze the boundary behavior of harmonic functions become crucual. Once the growth $u$ is controlled, along with the free boundary velocity\footnote{this is one of the places where a further assumption on $u$ is required, specifically a left-side time derivative bound see condition (1.1) in \cite[Theorem 1.2]{MR2603767}.}, one expects -and this is at a heuristic level- that the problem behaves a lot like the time-independent one-phase problem. Then,  $C^1$ regularity in  space and time of the free boundary is obtained by an iteration argument.

\textbf{Analyzing the free boundary via the hodograph transform, compactness and blow up arguments.} The ideas in the present work are closer in spirit to De Silva's work \cite{MR2813524} concerning the one-phase (time independent) problem with H\"older continuous coefficients. In \cite{MR2813524}, regularity of flat free boundary points is proved by a compactness argument and classification of blow up limits. The equation governing the blow up limits turns out to be a homogeneous Neumann problem for the Laplacian (which is known to be related to the time independent case of \eqref{eqn:intro frac heat equation}).

The approach in this paper focuses on the hodograph transform defined in Section \ref{sec:hodograph}. The hodograph transform is a well known tool in free boundary problems. A well known application of this transform is in the higher regularity theory for $C^{1}$ free boundaries by Kinderlehrer and Nirenberg \cite{Kinderlehrer1977regularity}, and more recently for lower dimensional obstacle problems by Koch, Petrosyan, and Shi \cite{Koch2015Higher}.  

For a problem evolving with time, a direct application of the improvement of flatness approach from \cite{MR2813524} to the hodograph transform presents several obstacles, making the recovery of $C^{1,\a}$ estimates more difficult. For instance, the scaling of the problem would require proving that the free boundary is $C^{1,\a}$ in time, which is false in general, as can be seen by the existence of spatially planar profiles with variable slopes \eqref{eqn:intro flat solutions} discussed earlier in the introduction. As mentioned earlier, being a solution to \eqref{eqn:intro frac heat equation} does not always guarantee interior regularity for $\p_t \bar u$: the nonlocal effects and Dirichlet data which is discontinuous in time immediately affects $\p_n \bar u$, making $\p_t \bar u$ discontinuous in the interior, see the example discussed in \cite[Section 6]{MR3148110}.

The idea of considering the equation for the gradient or a difference quotient of the free boundary is well known in the regularity theory of free boundary problems. It was used Caffarelli's work on two phase free boundary problems \cite{MR990856,MR973745}, in Athanasopoulos, Caffarelli, and Salsa's \cite{MR1394964,MR1397563,MR1486632} theory for two phase Stefan problems, and in the aforementioned works of Choi, Jerison, and Kim on the Hele-Shaw problem \cite{kim2006regularity,MR2218549,MR2306045, MR2603767}, among others. As it was also mentioned earlier, in \cite{MR2813524} De Silva obtained an improvement the free boundary for the one phase stationary problem, covering the case of operators in nondivergence form with H\"older continuous coefficients.

\textbf{Free boundary problems and integro-differential equations.} Given the widely known representation of the Dirichlet to Neumann map for the Laplacian as the fractional Laplacian, it should not be surprising that integro-differential techniques have something to say about boundary and free boundary problems. A prime example is the Signorini problem, known also as the thin obstacle problem, which is equivalent to the obstacle problem for the fractional Laplacian, as studied by Silvestre in \cite{Silvestre2007obstacle}, where regularity of the free boundary is obtained under a convexity condition. The two phase version of this problem is studied by Allen, Lindgren, and Petrosyan in \cite{Allen2015two}. We also mention work of the second author in collaboration with Schwab \cite{Guillen2014neumann,Guillen2015neumann}, where integro-differential techniques are used to study the homogenization of a boundary value problems, including linear Neumann problems with strong gradient dependence. 

The increasing number of readily available results for integro-differential equations further underlines the potential for applications to free boundaries. As highlighted earlier, our method requires results from the theory of parabolic integro-differential equations with bounded measurable coefficients. Such results -particularly H\"older estimates- can be found in work of the first author and D\'avila \cite{MR3148110}, as well as recent work of Schwab and Silvestre \cite{Schwab2014regularity}, the latter work considering regularity results for operators with non-symmetric kernels that are allowed to vanish in large portions of space. Likewise, it is worth mentioning work of Kassmann and Schwab \cite{Kassmann2014Parabolic}, which deals with divergence form operators but allows for some singular kernels (in this work, we always considered viscosity solutions, but it is conceivable to work from the beginning with weak solutions using a variational formulation). Methods from the theory of integro-differential equations appear in several other places in our proof. In Section \ref{sec:harnack}, we prove a Harnack estimate via a differential inequality argument similar to one used in Silvestre's work on critical Hamilton-Jacobi \cite{MR2737806}. Furthermore, in Section \ref{sec:improvement}, various ideas from work of Serra \cite{serra2014c,MR3385173} and from work of the first author and Kriventsov \cite{2015arXiv150406294C, 2015arXiv150507889C} were used to bootstrap a partial H\"older estimate up to a $C^{1,\alpha}$ estimate.

\subsection{Overview of the method} 

Let us describe further our strategy, highlighting the major steps. The first step is to extend the profile of the free boundary which is parametrized in terms of $\bar u$ by the hodograph transform of $u$. We consider, $\bar u:\bar B_1^+:=\{x\in B_1: x_n\geq0\}\to\R$ such that,
\begin{align*}
u(x-(t+\e\bar u(x,t))e_n,t) = x_n.
\end{align*}
The main idea is that $\bar u$ measures the horizontal distance between the graph of $u$ and the planar profile $(x_n+t)_+$ at scale $\e$. See Figure \ref{fig:hodograph} for a geometric description.

\afterpage{
\begin{figure}[t!]
\begin{center}
\includegraphics[width=14cm]{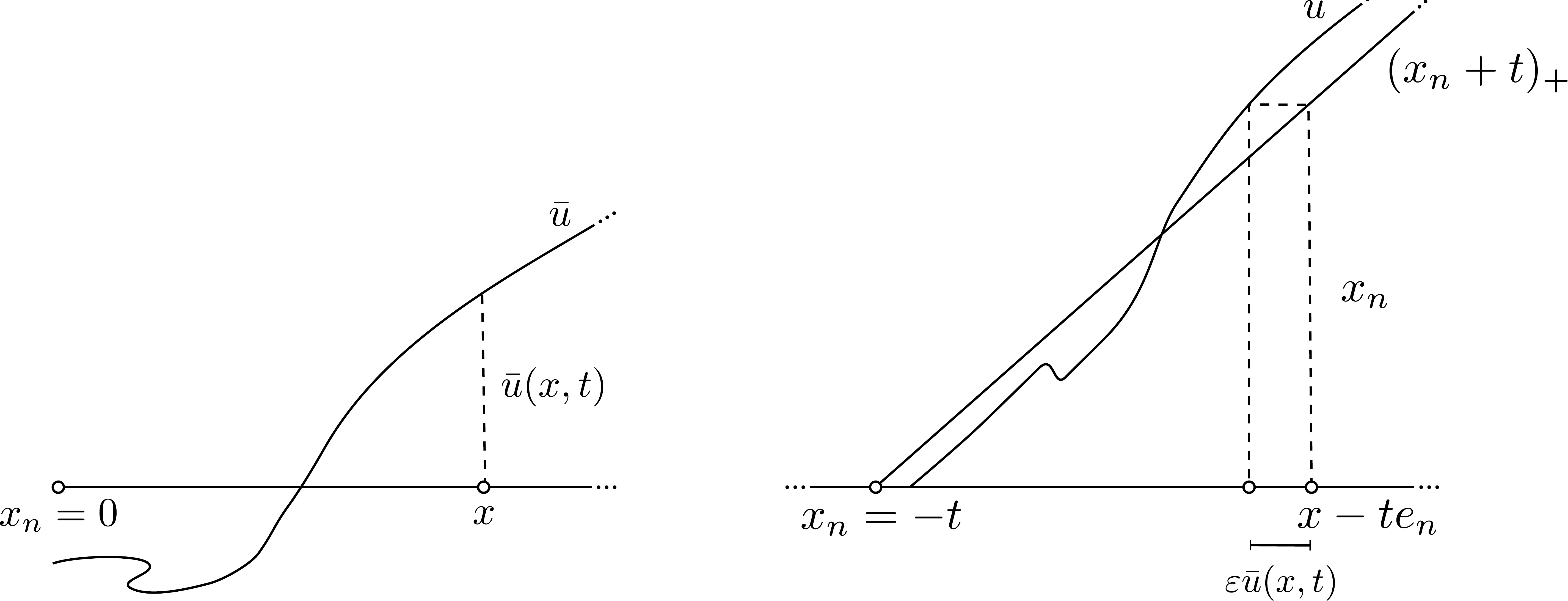}
\end{center}
\caption{Hodograph of $u$ with respect to $\e$ at a given time.}
\label{fig:hodograph}
\end{figure}
}

By implicit differentiation, we get that $\bar u$ satisfies some nonlinear relations depending on the flatness parameter $\e$ (see \eqref{eq:laplacian} and \eqref{eq:boundary_limit} respectively),
\begin{alignat*}{3}
\D \bar u &= \e F_\e(D^2\bar u, D\bar u) \qquad &&\text{ in } \qquad &&B_1^+ := \{x\in B_1:x_n>0\},\\
\partial_t \bar u &= \p_n \bar u + \e G_\e(D\bar u) \qquad &&\text{ in } \qquad &&B^{n-1}_1 := \{x \in B_1:x_n=0\}.
\end{alignat*}
As $\e\to0$ they linearize to,
\begin{align}\label{eq:linearized_intro}
\D\bar u = 0 \text{ in } B_1^+ \qquad \text{ and } \qquad \p_t \bar u = \p_n \bar u \text{ in } B^{n-1}_1.
\end{align}
In other words, $\bar u$ restricted to $B^{n-1}_1$ satisfies a nonlocal heat equation of order one. Moreover, a similar equation also holds for $D_{\R^{n-1}}\bar u$, which will imply the desired H\"older estimate.

We aim to establish an improvement of flatness by subsequently proving H\"older estimates for difference quotients of the form,
\[
\frac{\d_{he}\bar u}{h^\b}(x,t) := \frac{\bar u(x+(h/2)e,t)-\bar u(x-(h/2)e,t)}{h^\b}, \qquad \text{$h>0$, $e\in\p B_1^{n-1}$, $\b\in(0,1)$}.
\]
Where the exponent $\beta$ is improved by a fixed amount on every step. Leading to a $C^{1,\alpha}$ estimate after finitely many iterations. Several steps have to be settled in order to carry out this plan.

\textbf{Compactness:} In Section \ref{sec:harnack} we prove a Harnack type of estimate for sufficiently flat solutions. The ideas of the proof borrow significantly from the Harnack inequality argument used by De Silva in \cite{MR2813524} for the time-independent problem, and the Point Estimate for Hamilton-Jacobi equations with critical fractional diffusion used by Silvestre in \cite{MR2737806}. As a consequence, we prove the following: given $\e_k\to 0$, and sequence of solutions $u_k$ where $u_k$ is $\e_k$-flat, then the sequence $\d_{he}\bar u_k/h^\b$ has an accumulation point with respect to (local) uniform convergence for any $\b\in[0,\a)$, $h>0$, and $e\in\p B_1^{n-1}$.

\textbf{H\"older Bootstrap:} We reach $C^{1,\a}$ regularity by improving a finite number of times the exponent from a H\"older estimate for the solution. Ideally we would like to have the following implications for sufficiently flat solutions,
\begin{align}\label{eq:eta_zero}
\left\|\bar u\right\|_{L^\8(Q_1^{n-1})} + \sup_{h>0} \osc_{Q_1^{n-1}} \frac{\d_{he}\bar u}{h^\b} \leq 1 \qquad&\Rightarrow\qquad
\sup_{h>0} \left[\frac{\d_{he}\bar u}{h^\b}\right]_{C^{\a}\1Q_{1/4}^{n-1}\2} \leq C,\\
\nonumber &\Rightarrow\qquad \sup_{h>0} \osc_{Q_{1/16}^{n-1}} \frac{\d_{he}\bar u}{h^{\b+\a}} \leq C.
\end{align}
This result is easier to obtain if we allow the flatness to depend on $h$. However, this dependence could deteriorate as $h\to 0$. At this point the idea is to borrow some compactness from a previous H\"older estimate. In other words, we strengthen the hypothesis by considering $\eta\in(0,1)$ and (roughly) establishing that,
\begin{align*}
\left\|\bar u\right\|_{L^\8(Q_1^{n-1})} + \sup_{h>0}\left[\frac{\d_{he}\bar u}{h^\b}\right]_{C^\eta\1Q^{n-1}_1\2} \leq 1 \qquad&\Rightarrow\qquad
\sup_{h>0}\left[\frac{\d_{he}\bar u}{h^\b}\right]_{C^{\a+\eta}\1Q_{1/4}^{n-1}\2} \leq C,\\
&\Rightarrow\qquad \sup_{h>0} \left[\frac{\d_{he}\bar u}{h^{\b+\a}}\right]_{C^{\eta\a/4}\1Q_{1/16}^{n-1}\2} \leq C.
\end{align*}
The precise statement can be found in Lemma \ref{lem:improvement} and Corollary \ref{cor:improvement} which involve a different type of H\"older seminorms defined in the preliminary Section \ref{sec:prelim}. Theorem \ref{thm: main} follows from Corollary \ref{cor:main_thm}.

As an observation, notice that $\eta=0$ corresponds to the standard approach given by \eqref{eq:eta_zero}, however for $\eta>0$ we have the advantage that the uniform $C^\eta$ control now provides us with additional compactness. We use this to control the difference quotients when $h$ is arbitrarily small. This is one of the ideas that we learnt from recent estimates for nonlocal equations established by Serra in \cite{MR3385173, serra2014c}. Also recently, Kriventsov in collaboration with the first author, established time regularity estimates for parabolic problems using this technique in \cite{2015arXiv150406294C, 2015arXiv150507889C}.

The proof of Lemma \ref{lem:improvement} proceeds by a compactness argument. As $\e_k\to0$ and a rescaling of $\d_{he} \bar u_k$ converges to $w$ we recover that,
\begin{alignat*}{3}
\D w &= 0 \qquad &&\text{ in } \qquad &&\R^n_+ := \{x_n>0\},\\
\inf_{a\in[\l,\L]} a\p_n w \leq \p_t w &\leq \sup_{a\in[\l,\L]} a\p_n w \qquad &&\text{ in } \qquad &&\R^{n-1} := \{x_n=0\},
\end{alignat*}
where $0<\l\leq\L<\8$ depend only on the dimension. The main tool we use after this step is the Liouville's Theorem for fully nonlinear, nonlocal parabolic equations that results from a Harnack inequality. Such result can be found for instance in work of D\'avila and the first author \cite{chang2014h}, or in more recent work of Schwab and Silvestre \cite{Schwab2014regularity}.

\textbf{Limiting Equations:} In Section \ref{sec:improvement} we use that an accumulation point obtained as we send the flatness to zero satisfies the nonlocal heat equation. We prove this qualitative result in Section \ref{sec:jensen}. In other to reach this goal we consider a careful adaptation of the method of inf/sup convolutions used by Kim in order to establish the comparison principle for viscosity solutions of Hele-Shaw in \cite{MR1994745}. We will see that applying $\d_{he}/h^\b$ to the equations satisfied by $\bar u_k$ deteriorates the diffusion coefficient in terms of $|Du_k|$, which explains why we do not recover a constant coefficient equation for $w$.

One of the challenges of the outlined strategy comes from the scaling that corresponds to the H\"older bootstrap in Section \ref{sec:improvement}. The quantity we look to control is the difference quotient $\d_{he}\bar u_k/h^\b$ for which the appropriated scaling makes the oscillation of $\bar u_k$ to grow. In Section \ref{sec:jensen} it is important to keep in mind that in general $\bar u_k$ is not a compact sequence and only $\d_{he} \bar u_k$ is assumed to converge to $w$.

{\bf Acknowledgments.} Nestor Guillen was partially supported by the National Science Foundation, grant DMS-1201413. The authors would like to thank Inwon Kim, Ovidiu Savin and Daniela De Silva for many helpful discussions.

\section{Preliminaries}\label{sec:prelim}

In this section we set up some notation, define the notion of viscosity solutions, the hodograph transform and state the Liouville's Theorem for fully nonlinear, nonlocal parabolic equations.

\subsection{Notation}

$e_n$ denotes the $n^{th}$ vector of the canonical basis of $\R^n$.
\begin{alignat*}{3}
&&&\R^{n-1} := \{x\in\R^n:x_n=0\}\\
&&&\R^{n}_+ := \{x\in\R^n:x_n>0\}\\
&&&\bar \R^{n}_+ := \{x\in\R^n:x_n\geq 0\}\\
&&&B_{r_0}(x_0) := \{x\in\R^n:|x-x_0|<r_0\}\quad &&B_{r_0} := B_{r_0}(0)\\
&&&B_{r_0}^{n-1}(x_0) := B_{r_0}(x_0)\cap\R^{n-1}\quad &&B_{r_0}^{n-1} := B_{r_0}^{n-1}(0)\\
&&&B_{r_0}^+(x_0) := B_{r_0}(x_0)\cap\R^n_+\quad &&B_{r_0}^+ := B_{r_0}^+(0)\\
&&&\bar B_{r_0}^+(x_0) := B_{r_0}^+(x_0)\cup B_{r_0}^{n-1}(x_0)\quad &&\bar B_{r_0}^+ := \bar B_{r_0}^+(0)
\end{alignat*}
\begin{alignat*}{3}
&&&Q_{r_0}(x_0,t_0) := B_{r_0}(x_0)\times(t_0-r_0,t_0] \qquad &&Q_{r_0} := Q_{r_0}(0,0)\\
&&&Q_{r_0}^{n-1}(x_0,t_0) := B_{r_0}^{n-1}(x_0)\times(t_0-r_0,t_0] \qquad &&Q_{r_0}^{n-1} := Q_{r_0}^{n-1}(0,0)\\
&&&Q_{r_0}^+(x_0,t_0) := B_{r_0}^+(x_0)\times(t_0-r_0,t_0]\qquad &&Q_{r_0}^+ := Q_{r_0}^+(0,0)\\
&&&\bar Q_{r_0}^+(x_0,t_0) := \bar B_{r_0}^+(x_0)\times(t_0-r_0,t_0] \qquad &&\bar Q_{r_0}^+ := \bar Q_{r_0}^+(0,0)
\end{alignat*}

The last four sets are referred as \textit{parabolic cylinders of radius $r$ and centered at $(x_0,t_0)$}.

We use two different topologies for the time variable. The Euclidean one corresponding to the standard topology of $\R$ will be mostly assumed whenever we say that a given function is continuous or semicontinuous. The parabolic topology, where the family of intervals $\{(t_0-r,t_0]\}_{r>0}$ form a basis for the neighborhoods of $t_0$, will be used to establish H\"older regularity estimates. For instance, a H\"older modulus of continuity for $u$ at a point $(x_0,t_0)$ will be given by saying that the oscillation of $u$ in a parabolic cylinder of radius $r$ and centered at $(x_0,t_0)$ is controlled by $r^\a$ for some $\a\in(0,1]$. This is indeed equivalent to H\"older continuity in the Euclidean topology. The specific topology considered will be declared whenever is necessary.

The boundary operator $\p$ is always taken for a fixed time and with respect to the standard topology of the Euclidean space.

For an (open) domain $\W\ss\R^n$ and $u\in C(\W\to[0,\8))$ we define respectively the zero set, the positivity set, and the free boundary of $u$ as,
\begin{align*}
\W^0_u &:= \{x\in\W:u(x)=0\},\\
\W^+_u &:= \{x\in\W:u(x)>0\},\\
\G_u &:= \p \supp u \cap\W.
\end{align*}
Most of the time the domain $\W = \W(t)$ will depend on the time variable. In case this needs to be explicitly emphasized for the previous constructions we use,
\begin{align*}
\W^0_u(t) &:= \{x\in\W(t):u(x,t)=0\},\\
\W^+_u(t) &:= \{x\in\W(t):u(x,t)>0\},\\
\G_u(t) &:= \p\supp u(\cdot,t)\cap\W(t).
\end{align*}

\subsection{Viscosity Solutions}\label{sec:viscosity_sol}

For this section we use the Euclidean topology for the time variable and consider for $t\in(t_0,t_1]$, $\W(t)\ss\R^n$ an (open) domain such that $\p \W(t)$ is varying continuously in time in the sense of Hausdorff distance.

\begin{definition}[Speed of the interphase]
Let $r,h>0$, $\nu \in \p B_1$ and,
\begin{align*}
u &: C_{\nu,r}(x)\times(t-r,t]\to[0,\8),\\
C_{\nu,r,h}(x) &:= \{y\in\R^n: |(y-x) - ((y-x)\cdot\nu)\nu|<r \text{ and } |(y-x)\cdot\nu|<h\} \qquad \text{(cylinder)},
\end{align*}
such that there exists $\gamma:(B_r \cap span(\nu)^\perp)\times(-r,0] \to\R$ that parametrizes the free boundary of $u$ in the following way,
\begin{align*}
&\G_u(s) = \{y\in C_{\nu,r,h}(x): (y-x)\cdot\nu = \gamma((y-x) - ((y-x)\cdot\nu)\nu,t+s)\}.
\end{align*}
Then, if $\gamma$ is punctually first order differentiable at the origin we define the speed of the interphase at $x \in \G_u(t)$ by
\[
\frac{\p_t u}{|Du|}(x,t) := \p_t\gamma(0,0) = -\lim_{s\to0^-}\frac{\gamma(0,s)}{|s|}.
\]
\end{definition}

We will frequently use parametrizations of the form,
\[
\G_u(t) = \{x_n = -t-\e\bar u(x',t)\}.
\]
In this case, assuming enough regularity for $\bar u$, we obtain that,
\[
\frac{\p_t u}{|Du|} = \frac{1+\e \p_t\bar u}{\sqrt{1+\e^2|D\bar u|^2}}.
\]

\begin{definition}\label{def:gamma_reg}
Under the assumptions of the previous definition we say that:
\begin{enumerate}[label=(\alph*)]
\item $x\in \G_u(t)$ is a regular point in space and time if $\gamma$ is punctually $C^{1,1}$ at the origin. In other words, there exists $D\gamma(0,0)\in span(\nu)^\perp$ and $\p_t\gamma(0,0)\in\R$ such that,
\[
\gamma(x,t) = D\gamma(0,0)\cdot x+ \p_t\gamma(0,0)t + O(|x|^2+t^2).
\]
\item $\G_u \in C^1$ if $\gamma \in C^1$.
\end{enumerate}
\end{definition}

\begin{definition}[Free boundary relation]
Given $u(\cdot,t)\in C(\W(t)\to[0,\8))$ such that $\G_u\in C^1$ we say that,
\begin{align*}
\frac{\p_t u}{|Du|} \leq |Du| \text{ holds in the classical sense at $x\in\G_u(t)$},
\end{align*}
if $u(\cdot,t)\in C^1(\W_u^+(t)\cup\G_u(t))$ and
\[
\frac{\p_t u}{|Du|}(x,t) \leq |Du|(x,t) = \lim_{h\to0^+} \frac{u(x-h\nu,t)}{h}.
\]
\end{definition}

The relation $\frac{\p_t u}{|Du|} \geq |Du|$ in the classical sense is defined in a similar way.

\begin{definition}[Comparison Subsolution]
For $u(\cdot,t) \in C(\W(t)\to[0,\8))$ continuous in time we say that it is a comparison subsolution to the Hele-Shaw problem described in \eqref{eq:harmonic} and \eqref{eq:free_boundary} if:
\begin{enumerate}[label=(\alph*)]
\item For each $t\in(t_1,t_0]$, $u(\cdot,t)\in C^2(\W^+_u(t))$ and
\[
\D u \geq 0 \text{ in $\W^+_u$}.
\]
\item $\G_u \in C^1$, for each $t\in(t_1,t_0]$, $u(\cdot,t)\in C^1(\W^+_u(t)\cup\G_u(t))$, and
\begin{align*}
\frac{\p_t u}{|Du|} \leq |Du| \text{ holds in the classical sense for all $x\in\G_u(t)$.}
\end{align*}
\end{enumerate}
\end{definition}

We define the comparison supersolutions in a similar way by changing the direction of the inequalities above.

\begin{definition}[Contact]
Let $u$ be upper semicontinuous and $v$ lower semicontinuous defined in some common domain $\W\ss\R^n$. We say that $v$ touches $u$ from below at $x\in\W$ if $u(x)=v(x)$ and $u\geq v$ in some neighborhood of $x$.
\end{definition}

In the previous case we might also say that $u$ touches $v$ from above at the given point. This notion is also used in the case where both functions depend on the time variable.

\begin{definition}[Viscosity superharmonic functions]
We say that a lower semicontinuous function $u$ is a viscosity superharmonic function in $\W$ if whenever a smooth test function $v$ touches $u$ from below at $x$ we have that $\D w(x)\leq 0$. We denote it by,
\[
\D u \leq 0 \text{ in the viscosity sense over $\W$}.
\]
\end{definition}

We define a (upper semicontinuous) subharmonic function in a similar way by testing with smooth functions from above and changing the direction of the last inequality. Continuous functions that are both sub and superharmonic in the viscosity sense are harmonic functions in the classical sense.

\begin{definition}
We say that $u(\cdot,t)\in C(\W(t)\to[0,\8))$ has a continuously increasing support if $\supp u(\cdot, s)\ss \supp u(\cdot, t)$ for all $s<t$ and $\supp u(\cdot, t)$ varies continuously in time with respect to the Hausdorff distance.
\end{definition}

\begin{definition}[Viscosity Supersolution of Hele-Shaw]
Let $t\in[t_0,t_1]$, $\W(t)\ss\R^n$ an (open) domain such that $\p \W(t)$ is varying continuously in time in the sense of Hausdorff distance, and $u(\cdot,t)\in C(\W(t)\to[0,\8))$ lower semicontinuous in time with a continuously increasing support. Under these hypothesis we say that $u$ is a viscosity supersolution to the Hele-Shaw problem in the time interval $(t_0,t_1]$ if:
\begin{enumerate}[label=(\alph*)]
\item For each $t\in(t_1,t_0]$,
\[
\D u \leq 0 \text{ in the viscosity sense over $\W^+_u(t)$}.
\]
\item If $v$ is comparison subsolution, then $v$ can not touch $u$ from below at any $x \in \G_u(t)$.
\end{enumerate}
We denote the free boundary relation by,
\begin{align*}
\frac{\p_t u}{|Du|} \geq |Du| \text{ in the viscosity sense over $\G_u$}.
\end{align*}
\end{definition}

We define a viscosity subsolution in a similar way by requiring upper semicontinuity in time, subharmonicity in the positivity set, and that no comparison supersolution can touch $u$ from above at a free boundary point.

In order to define viscosity solutions that could be discontinuous in time we need to introduce the following upper a lower semicontinuous envelopes,
\begin{align}\label{eq:semi_cont_env}
u^*(x,t) =\limsup_{(y,s)\to (x,t)} u(y,s) \qquad \text{and} \qquad u_*(x,t) =\liminf_{(y,s)\to (x,t)} u(y,s).
\end{align}

\begin{definition}[Viscosity Solutions of Hele-Shaw]
Let $t\in(t_0,t_1]$, $\W(t)\ss\R^n$ an (open) domain such that $\p \W(t)$ is varying continuously in time in the sense of Hausdorff distance, and $u(\cdot,t)\in C(\W(t)\to[0,\8))$ lower semicontinuous in time with a continuously increasing support. Under these hypothesis we say that $u$ is a viscosity supersolution to the Hele-Shaw problem in the time interval $(t_0,t_1]$ if:
\begin{enumerate}[label=(\alph*)]
\item For each $t\in(t_1,t_0]$,
\[
\D u_* = \D u^* = 0 \text{ in the viscosity sense over $\W^+_u(t)$}.
\]
\item If $v$ is comparison subsolution, then $v$ can not touch $u_*$ from below at any $x \in \G_{u_*}(t)$.
\item If $v$ is comparison supersolution, then $v$ can not touch $u^*$ from above at any $x \in \G_{u^*}(t)$.
\end{enumerate}
We denote the free boundary relation by,
\begin{align*}
\frac{\p_t u}{|Du|} = |Du| \text{ in the viscosity sense over $\G_u(t)$}.
\end{align*}
\end{definition}

For lower semicontinuous functions, just the monotonicity of the support implies that $\supp u(\cdot, t)$ varies continuously in time, however this is not necessarily the case for upper semicontinuous functions. The continuity of the supports is what actually connects the solutions in time. Without it uniqueness fails as can be seen by considering an arbitrary solution in an interval $(t_0,t_1]$ and then just drastically changing the support immediately after time $t=t_1$. The continuity of $\supp u(\cdot, t)$ is indeed an important ingredient in the proof of the following comparison principle. The argument goes along the lines of Theorem 2.2 in \cite{salsa} by the continuity method.

\begin{property}[Comparison Principle]\label{pro:comparison}
Let $u$ be a viscosity supersolution of Hele-Shaw and $\varphi$ a comparison subsolution both defined in $\W(t)$ for $t\in[t_0,t_1]$ such that:
\begin{enumerate}[label=(\alph*)]
\item $u < \varphi$ in $\supp u(\cdot,t_0)$.
\item $u < \varphi$ on $\p\W(t) \cap \supp u(\cdot,t)$ for all $t\in[t_0,t_1]$.
\end{enumerate}
Then,
\[
u < \varphi \text{ in $\supp u(\cdot,t)$ for all $t\in[t_0,t_1]$}.
\]
\end{property}

A similar comparison holds between viscosity subsolution and comparison supersolutions. Notice that just a monotonicity hypothesis for the supports in the case of viscosity subsolutions will not give us the desired property. For instance, this will allow to change the solution by the harmonic replacement in $\W(t)$ at any given time.

As mentioned in the introduction, existence and uniqueness of viscosity solutions was established by Kim in \cite{MR1994745}. Existence followed by approximation with the porous medium equation as the power goes to infinity. Uniqueness was obtained by establishing a comparison principle between two viscosity solutions which turns out to be much more delicate than Property \ref{pro:comparison}. The main issue for such comparison principle is that usually one does not have enough regularity to evaluate the equation at a contact point between two free boundaries. The comparison proved in \cite{MR1994745} requires that the gradient of the functions over the initial free boundaries do not vanish. Without a similar hypothesis it is not expected to have uniqueness. Let us proceed to informally give an argument for such claim.

Consider for $n=2$ the case where the initial support of the solution forms an angle $\theta \in (0,\pi/2)$. Notice that the gradient of the harmonic function vanishes at the vertex like the distance to some power greater than one. In this case it was proved in \cite{EJM:2320324} that the vertex persists with the same angle for a positive amount of time. To construct a non unique problem one can look at an initial data formed by two copies of the previous example, joined by the vertices, as it is illustrated in Figure \ref{fig:non_uniqueness}.

\begin{figure}[t!]
\begin{center}
\includegraphics[width=14cm]{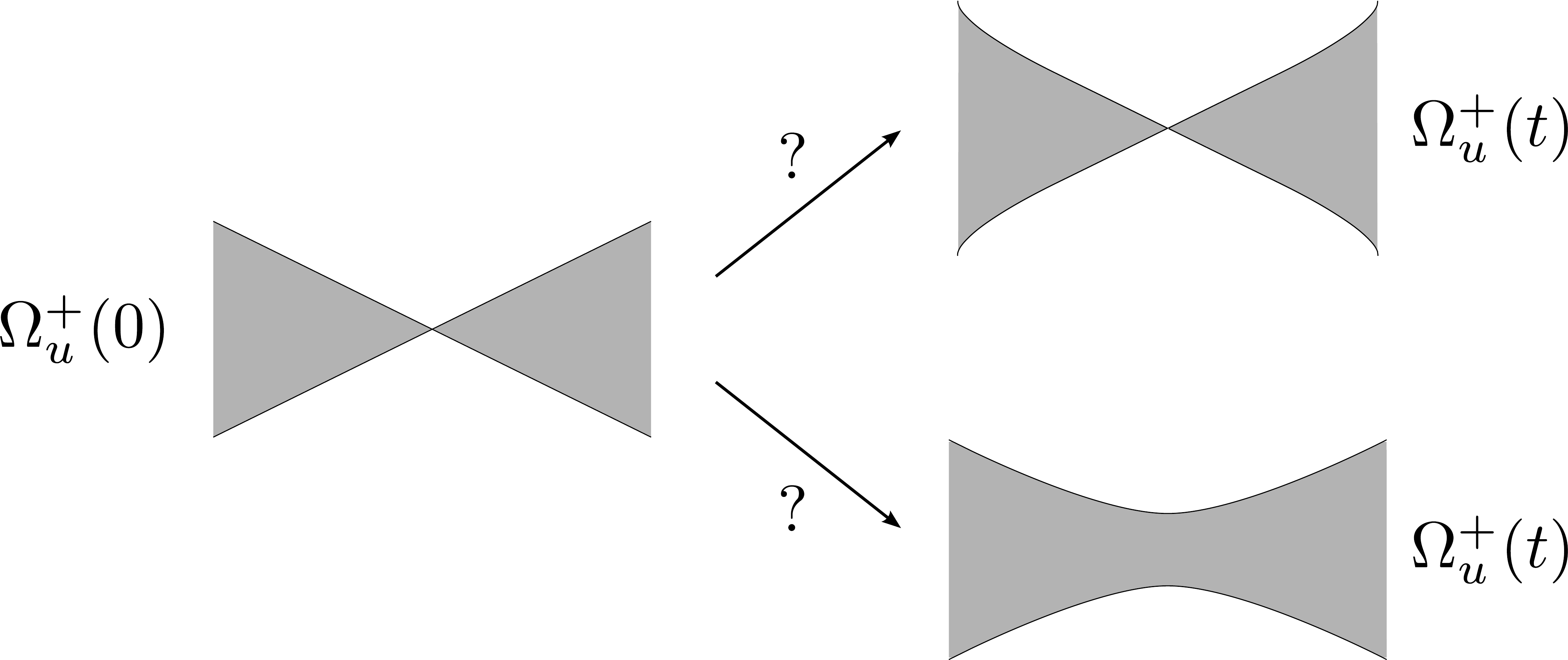}
\end{center}
\caption{}
\label{fig:non_uniqueness}
\end{figure}

Keep in mind that if we consider the separate evolution of the two supports we get that the vertex persists for a positive time. On the other hand, if we perturb the initial data by intersecting the supports in a non trivial way, then we expect that the obtuse angles now formed start to expand very fast. By taking the limit of these perturbations it is conceivable to recover a solution which merges and smooths out the vertices instantaneously.

\subsection{Hodograph}\label{sec:hodograph}

In most of the paper the domain $\W(t)$ will be equal to ball $B_r(-te_n)$ for some $r>0$ and $t\in(-T,0]$. Our main hypothesis is the closeness of a solution $u$ to the planar profile $(x_n+t)^+$. It turns out to be useful to measure such hypothesis in terms of the next construction. For the following definition $P(\R)$ is the set of all subsets of $\R$.

\begin{definition}[Hodograph]
We say that $\bar u:\bar\R^n_+\times(-\8,0] \to P(\R)$ is the hodograph transform of $u$ with respect to $\e>0$ if
\begin{align}\label{eq:hodo_def}
\bar u(x,t) := &\left\{ s\in\R: \exists \ (x_k,t_k)\in \R^n_+\times(-\8,0]\to (x,t), s_k\to s\right.\\
\nonumber &\left. \qquad\quad \text{ such that } u(x_k-(t_k+\e s_k)e_n,t_k) = (x_k)_n \right\}.
\end{align}
\end{definition}

Usually, instead of stating precisely \eqref{eq:hodo_def} we write the following more informal relation having in mind that $\bar u$ is a multi-valued function,
\[
u(x-(t+\e \bar u(x,t))e_n,t) = x_n.
\]
Figure \ref{fig:hodograph} illustrates the geometric meaning of the construction where $\bar u$ measures the horizontal distance between the graph of $u$ and the planar profile $(x_n+t)_+$ at scale $\e$. The approximating sequences in space are necessary in order to define $\bar u$ over $\R^{n-1}$ which relates with the free boundary of $u$ in the following way,
\[
\G_u(t) \ss \p\W_u^+(t) \cap \W(t) \ss \bigcup_{x'\in\R^{n-1}}\3x'-(t+\e\bar u(x',0,t))e_n\4.
\]

\subsubsection{Multi-valued functions}\label{sec:multi-valued_fn} A linear combination of multi-valued functions $v_1,v_2:\bar\R^n_+\times(-\8,0] \to P(\R)$ is a multi-valued function given by,
\[
(\a v_1 + \b v_2)(x,t) = \{\a a+\b b\in \R: a\in v_1(x,t), b\in v_2(x,t)\}.
\]
In particular, the centered difference of step size $h>0$ and in the direction of a unit vector $e$ is given by,
\begin{align*}
\d_{he} v(x,t) &:= v(x+(h/2)e,t) - v(x-(h/2)e,t).
\end{align*}

The oscillation and the $L^\8$ norm in a set $D \ss \bar\R^n_+\times(-\8,0]$ are defined as,
\begin{align*}
\osc_{D} v &:= \inf \3M \in [0,\8): \bigcup_{(x,t),(y,s)\in D} v(x,t)-v(y,s)\ss[-M,M]\4,\\
\|v\|_{L^\8(D)} &:= \inf \3M \in [0,\8):  \bigcup_{(x,t)\in D} v(x,t)\ss[-M,M]\4.
\end{align*}
We adopt the following convention whenever there exists $(x,t) \in D$ such that $v(x,t)=\emptyset$,
\begin{align*}
\osc_{D} v = \|v\|_{L^\8(D)} = \8.
\end{align*}
In other words, whenever one of our hypothesis says that one of the previous quantities is finite we are also assuming that $v(x,t)\neq \emptyset$ for any $(x,t) \in D$.

\begin{remark}
If $\bar u$ is the hodograph transform of $u$ with respect to $\e>0$, the hypothesis $\|\bar u\|_{L^\8\1Q^+_r(x_0,t_0)\2} \leq M <\8$ implies the following flatness hypothesis at each time $t\in(t_0-r,t_0] \ss (-\8,0]$ in $B_{r-\e M}(-te_n)$
\[
(x_n-(t+\e M))_+ \leq u(x,t)\leq (x_n-(t-\e M))_+.
\]
In a similar way, if the previous inequalities hold in $B_{r}(-te_n)$ then $\|\bar u\|_{L^\8\1Q^+_{r-\e M}(x_0,t_0)\2} \leq M$.
\end{remark}

Uniform convergence of a sequence $\{v_k\}_{k=1}^\8$ of (non empty) multi-valued functions towards a single-valued function $v$ gets derived from the previous construction,
\[
v_k \xrightarrow[k\to\8]{} v \qquad \text{uniformly in $D$ if}  \qquad \|v_k - v\|_{L^\8\1D\2} \xrightarrow[k\to\8]{} 0.
\]

We will define a H\"older semi-norm in terms of the following distance in $\bar\R^n_+\times(-\8,0]$
\begin{align*}
d((x,t),(y,s)) := \begin{cases}
|x-y| &\text{ if $s=t$},\\
\displaystyle\inf_{z,w\in\R^{n-1}}|x-z|+|(z,t)-(w,s)|+|w-y| &\text{ if $s\neq t$}.
\end{cases}
\end{align*}
The topology induced by this distance joins the different spatial domains $\R^n_+\times\{t\}$ across the boundary $\R^{n-1}\times(-\8,0]$. This is useful to address the fact that a function that is continuous under this topology is continuous in space up to the boundary $\R^{n-1}$ and continuous in space and time when restricted to $\R^{n-1}$.

We denote the parabolic space time balls with respect to $d$ by,
\[
B_r^d(x,t) := \{(y,s)\in\bar \R^n_+: d((x,t),(y,s)) < r\}, \qquad B_r^d:= B_r(0,0).
\]
Notice that if $x\in\R^{n-1}$ then,
\[
\bar Q_{r/\sqrt{n}}^+(x,t) \ss B_r^d(x,t) = \{(y,s)\in\bar \R^n_+: |(x,t)-(y,s)|<r, \ s\leq t\} \ss \bar Q_r^+(x,t).
\]

We define a truncated H\"older semi-norm for $D\ss\bar\R^n_+\times(-\8,0]$ with respect to $d$ by,
\begin{align*}
[v]^*_{C^\a_{trun(r)}(D)} &:= \inf \3M \in [0,\8): \bigcup_{\substack{(x,t),(y,s)\in D\\d((x,t),(y,s))>r}} \frac{v(x,t)-v(y,s)}{d((x,t),(y,s))^\a}\ss[-M,M]\4.
\end{align*}
As before, if there exists $(x,t)\in D$ such that $v(x,t) =\emptyset$, then we let $[v]^*_{C^\a_{trun(r)}(D)} = \8$.

The parameter $r$ denotes a truncation of the modulus of continuity. When $r=0$ and
\[
[v]^*_{C^\a(D)} := [v]^*_{C^\a_{trun(0)}(D)}<\8
\]
we have that $v$ is single-valued, continuous in space and continuous also in time when restricted to $D \cap \R^{n-1}$. Whenever $r=0$ and $D$ is a subset of $\R^{n-1}\times(-\8,0]$ or $\R^n\times \{t\}$, we have that our norm coincides with the standard H\"older seminorm and then we denote it by the standard notation by suppressing the star,
\[
[v]_{C^\a(D)} := [v]^*_{C^\a(D)}.
\]

\begin{remark}
Given the homogeneity of $d$, namely 
\[
d((\r x,\r t),(\r y,\r s)) = \r d((x,t),(y,s)),
\]
we get that for the rescaling $w(x,t):=v(\r x,\r t)$,
\[
[w]^*_{C^\a_{trun(\r r)}(\r D)} = \r^{\a}[v]^*_{C^\a_{trun(r)}(D)}.
\]
\end{remark}

\subsection{Integro-differential parabolic equations with bounded measurable coefficients}

After passing to the limit in our approximation arguments we find the following global problem,
\begin{alignat}{3}
\label{eq:heoo1}\D w &= 0 \qquad &&\text{ in } \qquad &&\R^n_+,\\
\label{eq:heoo2}\inf_{a\in[\l,\L]} a\p_n w &\leq \p_t w \leq \sup_{a\in[\l,\L]} a\p_n w \qquad &&\text{ in } \qquad &&\R^{n-1}\times(-\8,0].
\end{alignat}

Assuming that $w(\cdot,0,t)$ is sufficiently smooth about $x'\in\R^{n-1}$ and that for some $C>0$ and $\a\in(0,1)$,
\[
\|w(\cdot,t)\|_{L^\8(\bar B^+_R)} \leq CR^\a \text{ for every $R\geq 1$},
\]
we get the following well known fact from potential theory,
\begin{align*}
\p_n w(x',0,t) = \D^{1/2}_{\R^{n-1}}w(x',0,t) := C_n \lim_{\eta\to0}\int_{\R^{n-1}\sm B_\eta^{n-1}} \1w(x'+y',0,t)-w(x',0,t)\2\frac{dy'}{|y'|^n}
\end{align*}
where the constant $C_n>0$ is given such that we obtain the following identity for the Fourier multiplier,
\[
\widehat{\D_{\R^{n-1}}^{1/2}} = -|\xi|.
\]
This allows us to interpret \eqref{eq:heoo1} and \eqref{eq:heoo2} as the heat equation of order one with a bounded measurable diffusion.

\begin{definition}\label{def:viscosity_linear}
Let $0<\l\leq\L<\8$, $w:\R^n_+\times(-\8,0]\to\R$ continuous in space and time such that,
\begin{enumerate}[label=(\alph*)]
\item For every $t\leq 0$,
\[
\D w = 0 \text{ in } \R^n_+.
\]
\item There exists some $C>0$ and $\a\in(0,1)$ such that,
\[
\|w\|_{L^\8(\bar Q^+_R)} \leq CR^\a \text{ for every $R\geq 1$}.
\]
\end{enumerate}
Under these assumption we say that $w$ satisfies 
\[
\p_t w \leq \sup_{a\in[\l,\L]} a\p_n w \text{ in the viscosity sense in } \R^{n-1}\times(-\8,0]
\]
If whenever $\varphi \in C^\8(\bar Q_r^+(x_0',0,t_0))$ is a test function that touches $w$ from above at $(x_0',0,t_0) \in \R^{n-1}\times(-\8,0]$ we get that,
\begin{align*}
&\p_t \varphi(x_0',0,t_0) \leq \sup_{a\in[\l,\L]} a\p_n \varphi(x_0',0,t_0).
\end{align*}
\end{definition}

The inequality $\p_t w \geq \inf_{a\in[\l,\L]} a\p_n w$ in the viscosity sense is analogously defined by considering test functions touching $w$ from below, replacing the $\sup$ by the $\inf$, and changing the direction of the inequality. When restricted to $\R^{n-1}\times(-\8,0]$, viscosity solutions in the setting just described are also viscosity solutions in the sense of fully nonlinear, nonlocal parabolic equations considered in \cite{chang2014h}.

\begin{property}
Let $w$ satisfies,
\[
\p_t w \leq \sup_{a\in[\l,\L]} a\p_n w \text{ in the viscosity sense in } \R^{n-1}\times(-\8,0].
\]
Then for any $\varphi \in C^\8(Q_r^{n-1}(x_0',0,t_0))$ that touches $w(\cdot,0,\cdot)$ from above at $(x_0',0,t_0) \in \R^{n-1}\times(-\8,0]$ we get that,
\begin{align*}
&\p_t \phi(x_0',0,t_0) \leq \sup_{a\in[\l,\L]} a\D^{1/2}_{\R^{n-1}}\phi(x_0',0,t_0),
\end{align*}
where
\[
\phi(x',t) = \begin{cases}
\varphi (x',t) &\text{ if } (x',t) \in Q_r^{n-1}(x_0',t_0),\\
w(x',t) &\text{ otherwise}.
\end{cases}
\]
\end{property}

The proof consists on extending $\phi$ harmonically to $\R^n_+$ using that for $R\geq 1$, $\|\phi\|_{L^\8(Q^{n-1}_R)} \leq CR^\a$. Let $\psi$ be the unique extension which also satisfies $\|\psi\|_{L^\8(\bar Q^+_R)} \leq CR^\a$ for any $R\geq 1$. Finally we use that $\psi$ is a test function that touches $w$ from above such that $\p_n\psi(x_0',0,t_0) = \D^{1/2}_{\R^{n-1}}\phi(x_0',t_0)$.

A consequence of the Harnack inequality \cite[Corollary 6.4]{chang2014h} is the following Liouville theorem.

\begin{theorem}[Liouville's Theorem]\label{thm:liouville}
There exists $\a\in(0,1)$ sufficiently small depending on $0<\l\leq\L<\8$, such that if $w$ satisfies the assumptions of the Definition \ref{def:viscosity_linear} for such exponent $\a$ and
\[
\inf_{a\in[\l,\L]} a\p_n w \leq \p_t w \leq \sup_{a\in[\l,\L]} a\p_n w \text{ in the viscosity sense in } \R^{n-1}\times(-\8,0].
\]
Then $w$ is necessarily a constant function. 
\end{theorem}


\section{Harnack Inequality}\label{sec:harnack}

In this section we show that if a solution $u$ is sufficiently flat then the oscillation of $\bar u$ decreases in a smaller domain. Our main Theorem then says that we obtain a truncated H\"older modulus of continuity for $\bar u$ with respect to the distance $d$ introduced in Section \ref{sec:prelim}.

\begin{theorem}\label{thm:trun_holder}
Let $u(\cdot,t) \in C(B_1(-te_n)\to[0,\8))$ be a viscosity solution of the Hele-Shaw problem and $\bar u$ the hodograph transform of $u$ with respect to $\e>0$. There exists a H\"older exponent $\a\in(0,1)$, $\e_0\in(0,1)$, and $C>0$ such that,
\begin{align*}
\e \in(0,\e_0) \qquad\text{and}\qquad \|\bar u\|_{L^\8\1\bar Q^+_{7/8}\2}\leq 1 \qquad\Rightarrow\qquad [\bar u]^*_{C^\a_{trun(C\e)}\1B^d_{1/4}\2}\leq C.
\end{align*}
\end{theorem}

The main step to prove the previous Theorem relies on the following Harnack type estimate. The strategy is based on the Harnack inequality from \cite{MR2813524} and the Point Estimate for parabolic nonlocal equations in \cite{MR2737806}. Heuristically speaking, the speed of the interphase measured by $|Du|$ follows the behavior of $u$ at some point in $\W_u^+(t)$. This information can be integrated in time in order to see from which side the oscillation of the free boundary diminishes.

\begin{lemma}\label{lem:point_est}
Let $u(\cdot,t) \in C(B_1(-te_n)\to[0,\8))$ be a viscosity solution of the Hele-Shaw problem in the time interval $(-1,0]$ such that the following flatness hypothesis holds in $B_{3/4}(-te_n)$ for some $\e>0$, at each time $t\in(-3/4,0]$,
\[
(x_n+ t-\e)_+ \leq u(x,t) \leq (x_n+t+\e)_+.
\]
There exists $\e_0,\theta,\mu\in(0,1)$ such that if $\e\in(0,\e_0)$ then at least one of the following two holds in $\bigcup_{t\in(-\m,0]} B_\m(-te_n)$,
\begin{align*}
 u(x,t) \geq (x_n+t- (1-\theta)\e)_+, \qquad \text{or}\qquad
u(x,t)\leq (x_n+t+(1-\theta)\e)_+.
\end{align*}
\end{lemma}

\begin{proof}
Either,
\begin{align}\label{eq:density_hyp}
\frac{\left|\{t\in(-3/4,-1/2]: u((1/4-t)e_n,t)\geq 1/4\}\right|}{|(-3/4,-1/2]|} \geq 1/2,
\end{align}
or the opposite inequality is true. The treatment of these cases is very similar so we will just focus on the one stated above. Here our goal is to get the improvement of the oscillation from below,
\begin{align*}
u(x,t) \geq (x_n+t- (1-\theta)\e)_+.
\end{align*}
In order to do this it suffices to construct a comparison subsolution that can be used as a barrier.

Let $r:C^1([-3/4,0]\to[0,1))$ to be defined and set
\begin{align*}
p(t) := \frac{1}{8r(t)}-\frac{r(t)}{2},\qquad\text{ and }\qquad R(t) := \frac{1}{8r(t)}+\frac{r(t)}{2},
\end{align*}
such that $\p B_{R(t)}((p(t)-t+\e)e_n)$ is the unique sphere that contains the point $(\e-t-r(t))e_n$ and the $(n-2)$ dimensional sphere $\p B_{1/2}^{n-1}+(\e-t)e_n$. Consider the domain,
\begin{align*}
\W(t) := (B_R((p(t)-t+\e)e_n)\cap\{x_n \leq \e-t\})\cup(B_{3/4}(-te_n)\cap\{x_n> \e-t\}).
\end{align*}
In case $r(t)=0$, then $p(t)=R(t)=\8$, and $\W(t)=B_{3/4}(-te_n)\cap\{x_n> \e-t\}$.

\begin{figure}[t]
\begin{center}
\includegraphics[width=12cm]{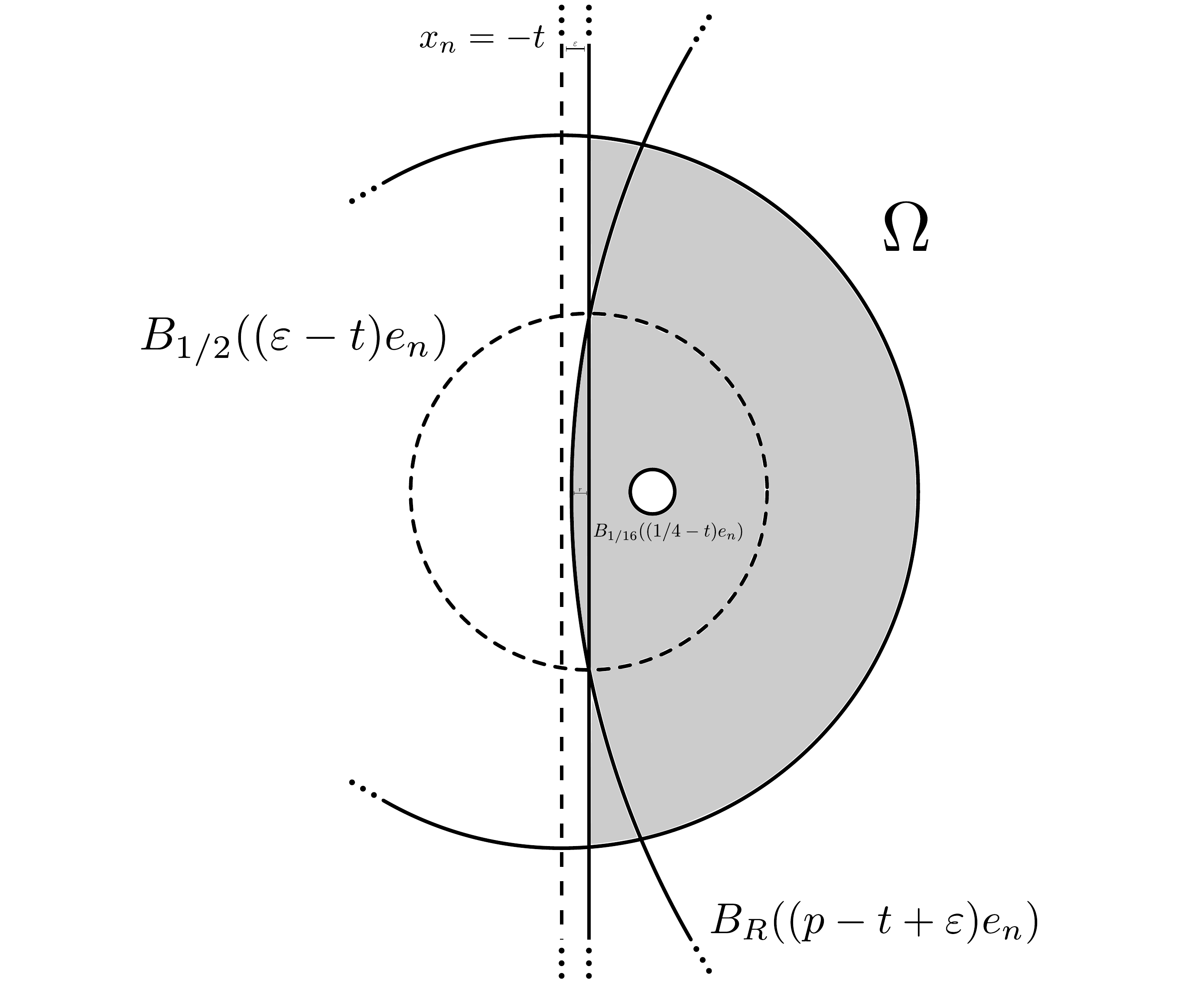}
\end{center}
\caption{Domain for the comparison subsolution $U+cV$ at a given time.}
\end{figure}

Let $U=U(\cdot,t),V=V(\cdot,t)$ such that,
\begin{alignat*}{2}
\D U &= \D V = 0 &&\text{ in } \W(t) \sm B_{1/16}((1/4-t)e_n),\\
U &= (x_n+t-\e)_+ &&\text{ on $\p(\W(t) \sm B_{1/16}((1/4-t)e_n))$},\\
V &= 0 &&\text{ on } \p \W(t),\\
V &= u((1/4-t)e_n)-(1/4-t-\e) &&\text{ on } \p B_{1/16}((1/4-t)e_n).
\end{alignat*}

By Harnack's inequality there exists a universal constant $c\in(0,1)$ such that,
\begin{align*}
u \geq U + cV \text{ in } \p B_{1/16}((1/4-t)e_n).
\end{align*}
In order for $(U + cV)$ to be a comparison subsolution in $\{x_n<\e-t\}$, it is sufficient to have that for every $x \in \p \W \cap\{x_n<\e-t\}$,
\begin{align}\label{sub}
1 + \p_t r \leq |DU| + c|DV|.
\end{align}

By Lemma \ref{lem:appendix} we get that there exist constants, $r_0\in(0,1)$ and $C>0$, depending only on the dimension, such that $r\in(0,r_0)$ implies that for every $x \in \p \W \cap\{x_n<\e-t\}$,
\begin{align*}
|DU| \geq 1 - Cr.
\end{align*}

On the other hand, by Hopf's Lemma and Harnack's inequality, we get that for every $t \in (-1/8,0]$ and $x \in \p \W \cap\{x_n<\e-t\}$,
\begin{align*}
&|DV| \geq c\1u((1/4-t)e_n)-(1/4-t-\e)\2 \geq c\e f(t),\\
&f(t):= \begin{cases}
1 &\text{ if $u((1/4-t)e_n,t)\geq 1/4$},\\
0 &\text{ otherwise}.
\end{cases}
\end{align*}

Now we fix $r(t)$ as the solution of the following initial value problem for which the forcing term records the density hypothesis \eqref{eq:density_hyp},
\begin{align*}
\begin{cases}
r' + Cr = c\e f(t),\\
r(-3/4) = 0.
\end{cases}
\end{align*}
This implies that \eqref{sub} gets satisfied provided that $r$ does not grow above $r_0$.

Integrating the differential equation,
\begin{align*}
r(t) = c\e \int_{-3/4}^t f(s)e^{-C(t-s)}ds \leq c\e,
\end{align*}
therefore we get that $r \in [0,r_0)$ if $\e \in(0,\e_0)$ is sufficiently small.

Using the density hypothesis \eqref{eq:density_hyp} we get that $r(t)\geq 4\theta\e$ for every $t\in[-1/2,0]$ and some $\theta\in(0,1)$ sufficiently small. This already gives us flatness for the interphase. Consider now $\widetilde \W(t)$ constructed as before with $r=4\theta\e$ fixed and $W=W(\cdot,t)$ such that,
\begin{alignat*}{2}
\D W &= 0 &&\text{ in } \widetilde\W(t) \sm B_{1/16}((1/4-t)e_n),\\
W &= (x_n+t-\e)_+ &&\text{ on $\p(\widetilde\W(t) \sm B_{1/16}((1/4-t)e_n))$}.
\end{alignat*}
Then, for every $t\in[-1/2,0]$, $u\geq U \geq W$. Using once again Lemma \ref{lem:appendix} we get that in $B_{1/4}(-te_n)$,
\[
W(x,t) \geq \int^0_{-(x_n+t-(1-2\theta)\e)_+} \p_n W(x+se_n,t)ds \geq (1-C\e)(x_n+t-(1-2\theta)\e)_+.
\]
This implies that if we take $\m$ as a sufficiently small multiple of $\theta$ then we recover from the inequality above the desired estimate in $B_\m(-te_n)$ for every $t\in(-\m,0]$,
\[
u \geq W\geq (1-C\e)(x_n+t-(1-2\theta)\e)_+ \geq (x_n+t-(1-\theta)\e)_+.
\]
\end{proof}

In terms of the hodograph $\bar u$, the previous Lemma says that at least one of the following two hold in $\bar Q^+_{\m-\e}$,
\[
-(1-\theta)\leq \bar u \qquad\text{or} \qquad \bar u \leq (1-\theta).
\]
Take $\e_0$ even smaller if necessary such that $\m-\e\geq\bar\m:=\m/2$. To iterate this decay of oscillation we consider the rescaling,
\begin{alignat*}{3}
v(x,t) &:= \frac{u(\bar\m x\pm(\e\theta/2)e_n, \bar\m t)}{\bar\m},\qquad &&\bar v(x,t) := \frac{\bar u(\bar\m x, \bar\m t)\pm\theta/2}{1-\theta}.
\end{alignat*}
The plus sign is chosen if $u \leq (1-\theta)$ in $\bar Q^+_{\bar\m}$ and the minus sign is chosen otherwise. We get that $v$ is still a solution of Hele-Shaw and $\bar v$ is the corresponding hodograph transform with respect to $\zeta:= (1-\theta)\e/\bar\m$. Therefore, the decay holds at scale $\bar\m$ if $\e\in(0,\bar\m\e_0)$. In general, the diminish of oscillation can be iterated $N$ times if $\e\in(0,\bar\m^N\e_0)$. The following Corollary then follows from this observation and a standard covering argument.

\begin{corollary}\label{cor:trun_holder}
Under the hypothesis of Theorem \ref{thm:trun_holder} and letting $\a = \ln(1-\theta)/\ln\bar\m$ there exists $\e_0\in(0,1)$ and $C>0$ such that,
\begin{align*}
\e \in(0,\e_0) \qquad\text{and}\qquad \|\bar u\|_{L^\8\1\bar Q^+_{7/8}\2}\leq 1 \qquad\Rightarrow\qquad \sup_{\substack{(x,t)\in Q^{n-1}_{1/4}\\\r\in(C\e,1/2)}} \r^{-a} \osc_{\bar Q_\r^+(x,t)} \bar u \leq C.
\end{align*}
\end{corollary}

By combining the previous corollary with the estimates for harmonic functions up to the boundary we get to establish the proof of Theorem \ref{thm:trun_holder}.

\begin{proof}[Proof of Theorem \ref{thm:trun_holder}]
Let $(x,t) = (x',x_n,t)\in B^+_{1/4}\cap\{x_n>4\e\}$ and $\r=x_n/4$. Given that $u(\cdot,t)$ is harmonic in $B_{3\r}(x-te_n)$ we get by the interior gradient estimate for $u(\cdot,t) - (x_n+t)^+$ that
\begin{align*}
\r\|Du(\cdot,t)-e_n\|_{L^\8\1B_{2\r}(x-te_n)\2} &\leq C\osc_{B_{3\r}(x-te_n)} (u-x_n),\\
&\leq C\e \osc_{\bar Q_{8\r}^+(x',0,t)} \bar u,\\
&\leq C\e\r^\a.
\end{align*}
Therefore, for $\e$ sufficiently small, $u$ is increasing in the $e_n$ direction, $\bar u(\cdot,t)$ is single valued in $B_\r(x) \supseteq B_{2\r-\e}(x)$, and by applying implicit differentiation to the relation,
\[
u(x-(t+\e\bar u(x,t))e_n,t) = x_n
\]
we obtain that,
\begin{align}\label{eq:interior_gradient}
\r\|D\bar u(\cdot,t)\|_{L^\8\1B_\r(x)\2} = \r\left\|\frac{1}{\e\p_n u}(Du-e_n)\right\|_{L^\8\1B_\r(x)\2} \leq C\r^\a.
\end{align}

Consider now two points $(x,t) = (x',x_n,t),(y,s)=(y',y_n,s) \in B^d_{1/4}$ such that $y_n\leq x_n$. In the following computations $\bar u(x,t)$, $\bar u(y,s)$, $\bar u(x',0,t)$ and $\bar u(y',0,s)$ denote arbitrary elements of the corresponding sets.

\textbf{Case I:} $4\e\geq x_n\geq y_n$. By Corollary \ref{cor:trun_holder},
\begin{align*}
|\bar u(y,s) - \bar u(x,t)| &\leq |\bar u(y,s) - \bar u(y',0,s)| +|\bar u(y',0,s) - \bar u(x',0,t)| +|\bar u(x',0,t) - \bar u(x,t)|,\\
&\leq C\1\e^\a+|(x',0,t)-(y',0,s)|^\a\2,\\
&\leq C\max\1\e,d((x,t),(y,s))\2^\a.
\end{align*}

\textbf{Case II:} $x_n>4\e$. Let $\r=x_n/4$. If $y\in B_\r^d(x)$ then $t=s$ and from \eqref{eq:interior_gradient}
\begin{align*}
\frac{|\bar u(y,s) - \bar u(x,t)|}{d((x,t),(y,s))^\a} \leq C\frac{d((x,t),(y,s))^{1-\a}}{\r^{1-\a}}\leq C.
\end{align*}
Otherwise, if $y\notin B_\r^d(x)$ then by Corollary \ref{cor:trun_holder},
\begin{align*}
|\bar u(y,s) - \bar u(x,t)| &\leq |\bar u(y,s) - \bar u(y',0,s)| +|\bar u(y',0,s) - \bar u(x',0,t)| +|\bar u(x',0,t) - \bar u(x,t)|,\\
&\leq C\1y_n^\a+|(x',0,t)-(y',0,s)|^\a+x_n^\a\2,\\
&\leq Cd((x,t),(y,s))^\a.
\end{align*}
This is the desired estimate which concludes the proof of the theorem.
\end{proof}


\section{H\"older Bootstrap}\label{sec:improvement}

In this section we establish the iterative procedure that allows us to recover the interior H\"older estimate for $D_{\R^{n-1}}\bar u$. For uniformly elliptic fully nonlinear equations, this can be done by assuming that the oscillation of the difference quotient $\d_{he}\bar u/h^\b$ is bounded. Using then that $\d_{he}\bar u/h^\b$ satisfies an equation with bounded measurable coefficients one then obtains a H\"older estimate for $\d_{he}\bar u/h^\b$ which can be used to improve the exponent. In this case we do not know that $\d_{he}\bar u/h^\b$ satisfies an equation with bounded measurable coefficients, however we expect that as $\e\to0$ the equation linearizes.

One of the challenges is to recover this argument uniformly in $h$. As we will see one of the main ideas of this inductive argument is to start with a hypothesis that controls a $C^\eta$ seminorm of the difference quotient. It turns out that this allows us to control the difference quotient for $h$ arbitrarily small by using some useful interpolation lemmas included in the appendix.

For the following Lemma, $\a\in(0,1)$ is a sufficiently small H\"older exponent such that Liouville's Theorem \ref{thm:liouville} and the H\"older estimate Theorem \ref{thm:trun_holder} hold for this given exponent. The constants $0<\l\leq \L<\8$ used in the next proof are also universal and determined in Section \ref{sec:jensen}.

\begin{lemma}\label{lem:improvement}
Let $u(\cdot,t) \in C(B_1(-te_n)\to[0,\8))$ be a viscosity solution of the Hele-Shaw problem in the time interval $(-1,0]$, $\bar u$ the hodograph transform of $u$ with respect to $\e>0$, and $e\in \p B_1^{n-1}$. Given $\b\in(0,1)$, $\eta \in (0,1-\b)$, $C_0>0$ and $r\in(0,1/2)$; there exists $\e_0\in(0,1)$ and $C>0$ depending on $\beta$, $\eta$, $C_0$, and $r$ such that if the following hypotheses are satisfied,
\begin{align*}
\e\in\10,\e_0\2,\qquad \|\bar u\|_{L^\8\1Q_1^+\2} \leq 1,\qquad \sup_{\substack{h \in \1C_0\e,r\2}}\left[\frac{\d_{he}\bar u}{h^\b}\right]^*_{C^\eta_{trun(C_0\e)}\1\bar B_r^d\2} \leq C_0,
\end{align*}
then,
\begin{align*}
\sup_{\substack{\rho\in(0,r/2)\\(x,t) \in B^d_{r/4}\\h\in(\rho^2,\rho)}} \rho^{-\a}\left[\frac{\d_{he} \bar u}{h^\b}\right]^*_{C^\eta_{trun(\rho^2)}\1 B^d_\rho(x,t)\2} \leq C.
\end{align*}
\end{lemma}

\begin{remark}
The previous estimate already implies that $\d_{he} \bar u$ and $\bar u$ are continuous (single-valued) functions with respect to the topology induced by the metric $d$.
\end{remark}

\begin{proof}
Assume by contradiction that for some sequence $\e_k\to 0^+$ there exists a sequence of solutions $u_k(\cdot,t)$ such that the hodograph transform $\bar u_k$ of $u_k$ with respect to $\e_k$ satisfies,
\begin{align*}
\|\bar u_k\|_{L^\8\1Q_1^+\2}\leq 1 \qquad \text{and}\qquad \sup_{\substack{h\in\1C_0\e_k,r\2}}\left[\frac{\d_{he}\bar u_k}{h^\b}\right]^{*}_{C^\eta_{trun(C_0\e_k)}\1 B_r^d\2} \leq C_0.
\end{align*}
However,
\[
\Theta(\r) := \sup_{\substack{\varrho\in(\r,r/2)\\(x,t) \in B^d_{r/4}}} \lim_{k\to\8} \sup_{\substack{h\in(\varrho^2,\varrho)}}\varrho^{-\a}\left[\frac{\d_{he} \bar u_k}{h^\b}\right]^*_{C^\eta_{trun(\varrho^2)}\1 B^d_\varrho(x,t)\2} \nearrow \8 \qquad \text{ as $\r\to0^+$}.
\]

Consider a sequence $\r_m\to0^+$ and let $\varrho_m \in (\r_m,r/2)$ and $(x_m,t_m) = (x_m',(x_m)_n,t_m) \in B^d_{r/4}$ such that,
\begin{align}\label{eq:Theta}
\lim_{k\to\8} \sup_{\substack{h\in\1 \varrho_m^2,\varrho_m\2}} \varrho_m^{-\a} \left[\frac{\d_{he} \bar u_{k}}{h^\b}\right]^*_{C^\eta_{trun(\varrho_m^2)}\1 B^d_{\varrho_m}(x_m,t_m)\2} \geq \frac{1}{2}\Theta\1\r_m\2\to\8.
\end{align}
Let $k_m$ sufficiently large such that,
\[
C_0\e_{k_m}<\varrho_m^{5/\a} \qquad\text{and}\qquad \sup_{\substack{h\in\1 \varrho_m^2,\varrho_m\2}} \varrho_m^{-\a} \left[\frac{\d_{he} \bar u_{k_m}}{h^\b}\right]^*_{C^\eta_{trun(\varrho_m^2)}\1B^d_{\varrho_m}(x_m,t_m)\2} \geq \frac{1}{4}\Theta\1\r_m\2.
\]
We get from the truncated $C^\eta$ control that,
\[
\sup_{\substack{h\in\1 \varrho_m^2,\varrho_m\2}} \left[\frac{\d_{he} \bar u_{k_m}}{h^\b}\right]^*_{C^\eta_{trun(\varrho_m^2)}\1B^d_{\varrho_m}(x_m,t_m)\2} \leq \sup_{\substack{h\in\1 C_0\e_{k_m},r\2}} \left[\frac{\d_{he} \bar u_{k_m}}{h^\b}\right]^{*}_{C^\eta_{trun(C_0\e_{k_m})}\1B^d_r\2} \leq C_0.
\]
Then \eqref{eq:Theta} implies that $\varrho_m\to0$.

After taking a subsequence we can assume one of the following two alternatives.

\textbf{Case I:} $(x_m)_n\leq \varrho_m$. Consider the following rescalings centered at $(x_m',0,t_m)$ for $t \in (-\mathcal R_m,0]$ where $\mathcal R_m:=r/(2\sqrt{n} \varrho_m) \to\8$,
\begin{alignat*}{3}
v_m(x,t) &:= \frac{u_{k_m}(\varrho_m x + x_m', \varrho_m t + t_m)}{\varrho_m},\quad &&\bar v_m(x,t) := \frac{\bar u_{k_m}(\varrho_m x + x_m', \varrho_m t + t_m)}{\varrho_m^{\b+\a}\Theta(\varrho_m)}.
\end{alignat*}
Therefore $v_m(\cdot,t)\in C(B_{\mathcal R_m}(-te_n)\to[0,\8))$ is also solution of Hele-Shaw and $\bar v_m$ is its hodograph transform with respect to $\zeta_m:= \e_{k_m}\varrho_m^{\b+\a-1}\Theta(\varrho_m)$. The hypotheses for $\bar u_{k_m}$ imply the following for any radius $R\geq 1$,
\begin{align}
\label{eq:flatness}&\zeta_m \|\bar v_m\|_{L^\8\1\bar Q_R^+\2}^{2/\a} \leq \e_{k_m}\varrho_m^{\b+\a-1}\Theta(\varrho_m)\1\frac{\|\bar u_{k_m}\|_{L^\8\1B_{\sqrt{n}\varrho_m R}^d(x_m',t_m)\2}}{\varrho_m^{\b+\a}\Theta(\varrho_m)}\2^{2/\a} \leq \e_{k_m}^{1/2},\\
\label{eq:compactness}&\sup_{\substack{h\in(\varrho_m R^2,R)}} \left[\frac{\d_{he} \bar v_m}{h^\b}\right]^*_{C^\eta_{trun(\varrho_m R^2)}\1 B_R^d\2} \leq R^\a,\\
\label{eq:liouville}&\sup_{\substack{h\in\1\varrho_m,1\2}} \left[\frac{\d_{he}\bar v_m}{h^\b}\right]^*_{C^\eta_{trun(\varrho_m)}\1 B_2^d\2} \geq \frac{1}{4}.
\end{align}
Indeed, \eqref{eq:compactness} gets deduced from the following computation,
\begin{align*}
\sup_{\substack{h\in(\varrho_m R^2,R)}} \left[\frac{\d_{he} \bar v_m}{h^\b}\right]^*_{C^\eta_{trun(\varrho_mR^2)}\1B^d_R\2} &\leq \sup_{\substack{h\in(\varrho_m^2 R^2, \varrho_m R)}} \frac{\varrho_m^{-\a}}{\Theta(\varrho_m)}\left[\frac{\d_{he} \bar u_{k_m}}{h^\b}\right]^*_{C^\eta_{trun(\varrho_m^2R^2)}\1 B_{\varrho_mR}^d(x_m',t_m)\2},\\
&\leq R^{\a}\frac{\Theta(\varrho_m R)}{\Theta(\varrho_m)},\\
&\leq R^{\a}.
\end{align*}
On the other hand, \eqref{eq:liouville} gets deduced from,
\begin{align*}
\sup_{\substack{h\in\1\varrho_m,1\2}} \left[\frac{\d_{he}\bar v_m}{h^\b}\right]^*_{C^\eta_{trun(\varrho_m)}\1B_2^d\2} &= \frac{\varrho_m^{-\a}}{\Theta(\varrho_m)}\sup_{\substack{h\in\1\varrho_m^2,\varrho_m\2}} \left[\frac{\d_{he}\bar u_{k_m}}{h^\b}\right]^*_{C^\eta_{trun(\varrho_m^2)}\1B_{2\varrho_m}^d(x_m',t_m)\2}\\
&\geq  \frac{\varrho_m^{-\a}}{\Theta(\varrho_m)}\sup_{\substack{h\in\1\varrho_m^2,\varrho_m\2}} \left[\frac{\d_{he}\bar u_{k_m}}{h^\b}\right]^*_{C^\eta_{trun(\varrho_m^2)}\1B_{\varrho_m}^d(x_m,t_m)\2} \geq \frac{1}{4}.
\end{align*}

From \eqref{eq:compactness} for $R=2$ and \eqref{eq:liouville}, Corollary \ref{lem:appendix3} implies that there exists $c>0$, depending on $(1-(\b+\eta))$ and $C_0$, but independent of $m$, such that for some $h_m \in (c,1)$
\begin{align*}
\left[\frac{\d_{h_m e}\bar v_m}{h_m^\b}\right]^*_{C^\eta_{trun(\varrho_m)}\1B_2^+\2} \geq \frac{1}{8}.
\end{align*}
Let us assume then without loss of generality that $h_m\to h \in [c,1] \ss (0,1]$. Moreover, after having fixed $h$ we can assume from \eqref{eq:compactness} that the following convergence holds locally uniformly over the boundary $\R^{n-1}\times(-\8,0]$,
\[
\d_{he}\bar v_m(\cdot,0,\cdot) - \d_{h e}\bar v_m(0,0,0) \to w(\cdot,0,\cdot),
\]
where for any $R\geq 1$,
\begin{align*}
\|w(\cdot,0,\cdot)\|_{L^\8\1Q^{n-1}_{R/\sqrt{n}}\2} \leq C_0R^\a.
\end{align*}

For fixed $t\leq 0$ and any subsequence $m_l$, the control given by \eqref{eq:compactness} in $\R^n_+$ implies that $\d_{he} \bar v_{m_l}(\cdot,\cdot,t)$ has an accumulation point which we also denote by $w(\cdot,\cdot,t)$, now extended to $\R^n_+$. By Lemma \ref{lem:lim_eq_elliptic} we get that $w(\cdot,\cdot,t)$ is harmonic, takes the boundary value $w(\cdot,0,t)$ an is sub-linear at infinity. By the uniqueness of solutions of the Dirichlet problem for the Laplace equation over $\R^n_+$ with sublinear growth at infinity and the arbitrariness of the subsequence, we recover that the original sequence $\d_{he} \bar v_{m}(\cdot,\cdot,t)$ has to converge to $w(\cdot,\cdot,t)$ locally uniformly in $\bar\R^n_+$. Moreover, we also get local uniform convergence in time. Indeed, let us assume by contradiction that there exists $R,\eta>0$ such that for some sequence $t_m \to t_\8 \in [-R,0]$,
\[
\osc_{\bar B_R^+}  \1\d_{he}\bar v_m(\cdot,\cdot,t_m) - \d_{h e}\bar v_m(0,0,0) - w(\cdot,\cdot,t_m) \2 > \eta
\]
By compactness we can assume that $(\d_{he}\bar v_m(\cdot,\cdot,t_m) - \d_{h e}\bar v_m(0,0,0))$ convergences locally uniformly. By Lemma \ref{lem:lim_eq_elliptic} we get that the accumulation point has to be a harmonic function. By the local uniform convergence over $\R^{n-1}\times(-\8,0]$ we get that such harmonic function has to be $w(\cdot,\cdot,t_\8)$. This now contradicts the fact that the previous oscillation over $\bar B_R^+$ was uniformly greater than $\eta>0$.

From \eqref{eq:liouville} and the lower bound for $h$ we get that,
\begin{align}
\label{eq:liouville_limit}
[w]_{C^\eta(B^d_2)} > 0.
\end{align}

From \eqref{eq:flatness} and the previous analysis about the convergence of $\d_{he} \bar v_m \to w$, Theorem \ref{thm:lim_eq} implies that $w$ satisfies the following viscosity relations in $\R^{n-1}\times(-\8,0]$
\begin{align*}
\inf_{a\in[\l,\L]} a\p_n w\leq \p_tw \leq \sup_{a\in[\l,\L]} a\p_n w.
\end{align*}
Given that $w$ is sublinear at infinity, Liouville's Theorem \ref{thm:liouville} implies that $w$ is constant therefore contradicting \eqref{eq:liouville_limit}. This concludes Case I.

\textbf{Case II:} $(x_m)_n > \varrho_m$. In this case we now consider the rescalings centered at $(x_m,t_m)$
\begin{alignat*}{2}
v_m(x,t) &:= \frac{u_{k_m}(\varrho_m x + x_m, \varrho_m t + t_m)}{\varrho_m},\qquad &&\bar v_m(x,t) := \frac{\bar u_{k_m}(\varrho_m x + x_m, \varrho_m t + t_m)}{\varrho_m^{\b+\a}\Theta(\varrho_m)}.
\end{alignat*}
Similar to the previous case we find that for some $h>0$, $\d_{he}\bar v_m(\cdot,0)$ has an accumulation point $w \in C(\{x_n>-1\}\to\R)$ which is harmonic, sublinear at infinity and satisfies $[w]_{C^\eta(B_1)} > 0$. Notice the analysis now happens for a fixed time, taking the form of an elliptic problem. Again we obtain that $w$ provides a contradiction to Liouville's Theorem (just for the Laplace equation) from where we conclude Case II and the proof of the lemma.
\end{proof}

\begin{corollary}\label{cor:improvement}
Under the same assumptions of Lemma \ref{lem:improvement}: If $\b+\a+\eta<1$ then for some $C>0$ depending on $\beta$, $\eta$, $C_0$, and $r$,
\begin{align}\label{eq:less_than_one}
\sup_{h\in\1C\e,r/16\2}\left[\frac{\d_{he}\bar u}{h^{\b+\a}}\right]^*_{C^{\eta\a}_{trun\1C\e\2}\1B^d_{r/16}\2} \leq C.
\end{align}
If $\b+\a>1$ then the directional derivative $\p_e\bar u := e\cdot D\bar u$ exists and satisfies,
\begin{align}\label{eq:bigger_than_one}
[\p_e\bar u]_{C^\a\1Q_{r/(4\sqrt{n})}^{n-1}\2} \leq C.
\end{align}
\end{corollary}

\begin{proof}
Setting $h=\r$ in Lemma \ref{lem:improvement} we get that,
\[
\sup_{h\in(0,r/8)} \left\|\frac{\d_{he}^2\bar u}{h^{\b+\a+\eta}}\right\|_{L^\8\1 B_{r/16}^d\2} \leq C.
\]
If $\b+\a+\eta<1$ then Lemma \ref{lem:interpolation} implies that,
\[
\sup_{h\in(0,r/8)} \left\|\frac{\d_{he}\bar u}{h^{\b+\a+\eta}}\right\|_{L^\8\1 B_{r/16}^d\2} \leq C.
\]
Our goal is to get that,
\[
\sup_{\substack{h,\r \in \1C\e,r/8\2\\(x_0,t_0)\in B^d_{r/16}}} \osc_{B^d_\r(x_0,t_0)} \frac{\d_{he}\bar u}{h^{\b+\a}\r^{\eta\a}} \leq C.
\]

We consider two cases. If $\r^{\a} \geq h$ we bound $\d_{he}\bar u$ in terms of $h$,
\[
\osc_{Q_\r^{n-1}(x_0,t_0)} \frac{\d_{he}\bar u}{h^{\b+\a}\r^{\eta\a}} \leq C\frac{h^\eta}{\r^{\eta\a}} \leq C.
\]
If $\r^{\a} \in ((C\e)^{\a},h)$ we bound $\d_{he}\bar u$ in terms of $\r$ using the truncated H\"older estimate given by Theorem \ref{thm:trun_holder},
\[
\osc_{Q_\r^{n-1}(x_0,t_0)} \frac{\d_{he}\bar u}{h^{\b+\a}\r^{\eta\a}} \leq \frac{\r^{\a-\eta\a}}{h^{\beta+\a}} \leq C.
\]
This concludes the case $\b+\a+\eta<1$.

If $\b+\a>1$ then Lemma \ref{lem:appendix4} implies that $\p_e\bar u$ exists and satisfies \eqref{eq:bigger_than_one}.
\end{proof}

The following corollary provides the last step proving Theorem \ref{thm: main}. In the following we let $M$ be a positive integer sufficiently large such that the previous results hold for $\a = 1/M$.

\begin{corollary}\label{cor:main_thm}
Let $u(\cdot,t) \in C(B_1(-te_n)\to[0,\8))$ be a viscosity solution of the Hele-Shaw problem in the time interval $(-1,0]$ and $\bar u$ the hodograph transform of $u$ with respect to $\e>0$. There exist $\e_0\in(0,1)$, $r\in(0,1)$ and $C>0$ universal such that,
\[
\e\in\10,\e_0\2, \qquad \|\bar u\|_{L^\8\1\bar Q_1^+\2} \leq 1 \qquad \Rightarrow \qquad [D_{\R^{n-1}}\bar u]_{C^{\a/2}\1Q_r^{n-1}\2} \leq C.
\]
\end{corollary}

\begin{proof}
Let $e\in \p B^{n-1}_1$, $\b = 1-\a/2$ and for $k=0,1,2,\ldots,(M-2)=(1/\a-2)$,
\[
\b_k := \a/2+k\a,\qquad \eta_k := (1-\b)\a^k,\qquad r_k:=16^{-(k+1)}.
\]
Our first claim is that for each $k=0,1,2,\ldots,(M-1)$, there exists $\e_k\in(0,1)$ and $C_k>0$ such that,
\begin{align}\label{eq:inductive_hyp}
\e\in\10,\e_k\2, \qquad \|\bar u\|_{L^\8\1\bar Q_1^+\2} \leq 1 \qquad\Rightarrow\qquad \sup_{h\in\1C_k\e,r_k\2}\left[\frac{\d_{he}\bar u}{h^{\b_k}}\right]_{C^{\eta_k}_{trun\1C_k\e\2}\1B^d_{r_k}\2} \leq C_k.
\end{align}

Theorem \ref{thm:trun_holder} provides the result for $k=0$. Let us assume that \eqref{eq:inductive_hyp} holds for some $k<(M-2)$ and let $\bar\e_k$ and $\bar C_k>0$ be the constants resulting from applying Corollary \ref{cor:improvement} with respect to $\b_k$, $\eta_k$, $C_k$ and $r_k$. Then, the inductive step holds for $\e_{k+1} = \min(\e_k,\bar\e_k)$ and $C_{k+1}:=\max(C_k,\bar C_k)$.

For the final step we consider $\bar\e$ and $\bar C>0$ be the constants resulting from applying Corollary \ref{cor:improvement} with respect to $\b_{M-1} = 1-\a/2$, $\eta_{M-1}$, $C_{M-1}$ and $r_{M-1}$ and $\e_0 := \min(\e_{m-1},\bar\e)$ and $C :=\max(C_{m-1},\bar C)$. The estimate then gets established for the directional derivative $\p_e\bar u$ after applying Lemma \ref{lem:improvement} now in the case where $\b+\a>1$. The arbitrariness of the direction $e\in \p B^{n-1}_1$ settles the proof.
\end{proof}

\section{Limiting Equations}\label{sec:jensen}

Our goal in this section is to recover a limiting equation for the difference quotient of the free boundary as the flatness goes to zero. Assuming for a moment that the Hele-Shaw equations for $u$ hold in the global domain $\R^n\times(-\8,0]$ and that $\bar u$ is single valued and smooth, we get to differentiate the relation,
\[
u(x-(t+\e\bar u(x,t))e_n,t)= x_n.
\]
From the free boundary condition we get that $\bar u$ satisfies in $\R^{n-1}\times(-\8,0]$,
\begin{align}\label{eq:boundary_limit}
\underbrace{\frac{1+\e\p_t\bar u}{\sqrt{1+\e^2|D_{\R^{n-1}}\bar u|}}}_{\frac{\p_t u}{|Du|}} = \underbrace{\frac{\sqrt{1+\e^2|D_{\R^{n-1}}\bar u|}}{1-\e\p_n\bar u}}_{|Du|}\qquad \Rightarrow \qquad \partial_t \bar u = \frac{\p_n \bar u}{1-\e\p_n \bar u} + \e\frac{|D_{\R^{n-1}}\bar u|^2}{1-\e\p_n \bar u}.
\end{align}
Consider now the difference $w(x,t) = \d_{h e}\bar u(x,t)$, where $h>0$ and $e\in\p B_1^{n-1}$. Then,
\begin{align*}
\partial_t w = \d_{he}\1\frac{\p_n \bar u}{1-\e\p_n\bar u}\2 + \e  \d_{h e}\1\frac{|D_{\R^{n-1}}\bar u|^2}{1-\e\p_n \bar u}\2 = a\p_n w + \e \d_{h e}\1\frac{|D_{\R^{n-1}}\bar u|^2}{1-\e\p_n \bar u}\2,
\end{align*}
where
\begin{align*}
a(x,t) = \left.\frac{1}{1-\e\p_n\bar u}\right|_{(x,t)}+\e\1\left.\frac{1}{1-\e\p_n\bar u}\right|_{(x+(h/2)e,t)}\2\1\left.\frac{1}{1-\e\p_n\bar u}\right|_{(x-(h/2)e,t)}\2
\end{align*}
In order to obtain a uniformly elliptic equation as the flatness $\e\to0$ it is desirable to have
\[
0 < \l \leq \frac{1}{1-\e\p_n\bar u} = |Du| \leq \L < \8.
\]
We will see that these bounds can be enforced at regular points of the free boundary by combining the flatness hypothesis with the standard barrier argument used in the Hopf Lemma.

Here is the main result of this section. The exponent $\a$ is the one from Theorem \ref{thm:trun_holder}.

\begin{theorem}\label{thm:lim_eq}
Let $\e_k\to0^+$, $\mathcal R_k\to\8$, $u_k(\cdot,t) \in C(B_{\mathcal R_k}(-te_n)\to[0,\8))$ be a sequence of viscosity solutions of the Hele-Shaw problem in the time interval $(-\mathcal R_k,0]$, $\bar u_k$ its hodograph transform with respect to $\e_k$, $e\in\p B_1^{n-1}$ and $h>0$ such that,
\begin{alignat*}{2}
&\e_k^{\a/2}\|\bar u_k\|_{L^\8\1Q_R^+\2} \xrightarrow[k\to\8]{} 0.\\
&\d_{he}\bar u_k \xrightarrow[k\to\8]{} w \quad &&\text{locally uniformly in $\R^n_+\times(-\8,0]$},\\
&\|w\|_{L^\8(Q_R^+)} \leq CR^\a.
\end{alignat*}
then $w(\cdot,t)$ is a harmonic function in $\R^n_+$ and satisfies the following global viscosity relations in $\R^{n-1}\times (-\8,0]$,
\begin{align*}
\inf_{a\in[\l,\L]} a\p_n w \leq \p_t w \leq \sup_{a\in[\l,\L]} a\p_n w
\end{align*}
for some ellipticity constants $0<\l\leq\L<\8$ depending only on $n$.
\end{theorem}

The following Lemma shows that for $t$ fixed $w(\cdot,t)$ is harmonic. We assume without loss of generality that $t=0$ and ignore the time dependence. Here the hodograph $\bar u$ is constructed as in Section \ref{sec:prelim} but keeping $t=0$ fixed. This gives a multivalued function for which each one of its set values is a subset of the original hodograph defined using approximating sequences in space and \textit{time}. The hypothesis assumed in Theorem \ref{thm:lim_eq} are inherited for this construction.

\begin{lemma}\label{lem:lim_eq_elliptic}
Let $\e_k\to0$, $\mathcal R_k\to\8$, $u_k\in C(B_{\mathcal R_k}\to [0,\8))$ harmonic in $B_{\mathcal R_k}\cap\{u_k>0\}$ and $\bar u_k$ the hodograph transform with respect to $\e_k$
\begin{align}
\label{eq:implicit_hodograph}
u_k(x-\e_k\bar u_k(x)e_n) = x_n;
\end{align}
let $e\in\p B_1^{n-1}$ and $h>0$.

Given that
\begin{align*}
&\e_k^{1/2}\|\bar u_k\|_{L^\8\1B_R^+\2} \xrightarrow[k\to\8]{} 0 \qquad \text{ for any $R>0$}.\\
&\d_{he}\bar u_k \xrightarrow[k\to\8]{} w \qquad \text{locally uniformly in $\R^n_+$},
\end{align*}
then $w$ is a harmonic function in $\R^n_+$.
\end{lemma}

\begin{remark}For the next proof we use the following notation for $e\in\p B_1^{n-1}$ and $h>0$, 
\[
\mathcal B_{\r}(X_0) := B_{\r}(X_0)\cup B_{\r}(x_0+(h/2)e)\cup B_{\r}(X_0-(h/2)e).
\]
\end{remark}

\begin{proof}
Let us fix $x_0\in\R^n_+$, $\r=(x_0)_n/5$, $R=10\r$, and
\[
X_0 = x_0 -\e \bar u_k(x_0)e_n.
\]
For $k$ sufficiently large $\r>\e_k\|\bar u_k\|_{L^\8\1B_R^+\2}$, which implies $\mathcal B_{4\r}(X_0)\ss \{u_k>0\}$. Then $u_k$ is harmonic in $\mathcal B_{4\r}(X_0)$ and
\begin{align*}
\r\|Du_k-e_n\|_{L^\8(\mathcal B_{3\r}(X_0))} + \r^2\|D^2u_k\|_{L^\8(\mathcal B_{3\r}(X_0))} \leq C\e_k\|\bar u_k\|_{L^\8\1B_R^+\2}.
\end{align*}
For $k$ sufficiently large we get that $u_k$ is increasing in the $e_n$ direction, therefore $\bar u_k$ is single valued and smooth in $\mathcal B_{2\r}(x_0)$. By implicitly differentiating \eqref{eq:implicit_hodograph} we obtain,
\begin{align*}
\D_{\R^{n-1}}u_k &= \frac{1}{\p_n u_k}\12 D_{\R^{n-1}} \p_n u_k - \frac{\p_n^2u_k}{\p_n u_k} D_{\R^{n-1}} u_k\2\cdot D_{\R^{n-1}} u_k + \e_k\p_n u_k\D_{\R^{n-1}}\bar u_k,\\
\p_n^2u_k &= \e_k(\p_nu_k)^{3}\p_n^2\bar u_k.
\end{align*}
Form the harmonicity of $u_k$,
\begin{align}
\label{eq:laplacian} &\D\bar u_k+b_k(x)\D_{\R^{n-1}}\bar u_k= f_k(x),\\
\nonumber &b_k(x) := \left.\1\1\frac{1}{\p_n u_k}\2^{2}-1\2\right|_{x-\e_k\bar u_k(x)e_n},\\
\nonumber &f_k(x) := \left.\1-\frac{1}{\e_k\p_n u_k^4}\12 D_{\R^{n-1}} \p_n u_k - \frac{\p_n^2u_k}{\p_n u_k} D_{\R^{n-1}}u_k\2\cdot D_{\R^{n-1}} u_k\2\right|_{x-\e_k\bar u_k(x)e_n}.
\end{align}
Taking $\d_{he}$ we obtain the following relation in $B_{2\r}(x_0)$,
\begin{align*}
&\D\1\d_{he} \bar u_k\2+c_k(x)\D_{\R^{n-1}}\1\d_{he} \bar u_k\2= F_k(x)\\
&c_k(x) := b_k(x+(h/2)e)\\
&F_k(x) := \d_{he} f_k(x)-\1\d_{he} b_k(x)\2\D_{\R^{n-1}}\bar u_k(x-(h/2)e)
\end{align*}
By the interior estimates for $u$ in $\mathcal B_{3\r}(X_0)$ we know that the coefficients above satisfy,
\begin{align*}
&\|c_k\|_{L^\8(B_{\r}(x_0))} + \|F_k\|_{L^\8(B_{\r}(x_0))} \leq \xi_k := C(\r)\e_k\1\|\bar u_k\|_{L^\8\1B_R^+\2}^2+1\2 \xrightarrow[k\to\8]{} 0.
\end{align*}
Then,
\begin{align*}
|\D\1\d_{he}\bar u_k\2| - \xi_k|\D_{\R^{n-1}}\1\d_{he}\bar u_k\2| \leq \xi_k.
\end{align*}
The harmonicity of $w$ in $B_\r(x_0)$ now follows from the standard stability for elliptic equations.
\end{proof}

The challenging part of Theorem \ref{thm:lim_eq} is obtaining the linearization of the free boundary relation. Unlike the previous lemma, there may not be sufficient regularity of the solution that would allow us to evaluate the equations for $\bar u$ at every point, which complicates obtaining an equation for the difference of the solution and its translate. This is a delicate issue appearing throughout the theory of viscosity solutions that fortunately has been well understood since work of Jensen \cite{jensen}. The idea is to approximate the solution by inf/sup convolutions which enjoy the following useful properties:
\begin{enumerate}[label=(\alph*)]
\item They become sub and supersolutions.
\item They allow to evaluate the free boundary condition in a classical sense at suitable regular points, which are a set of full measure.
\item They approximate the original solution.
\end{enumerate}
An additional advantage that we obtain with the construction is that it allows to control also the ``curvature'' of the (regularized) free boundaries, which yields control over $|Du|$.

\subsection{Inf and Sup Convolutions}\label{sec:inf_conv}

Let $\e\in(0,1)$, $R\geq 1$ and $\mathcal{R} > R$ be a sufficiently large radius that will change in each statement in order to allow enough room for the constructions and results of this section. The reader should keep in mind that once a property gets established for $R$ then those attributes can be immediately assumed for $\mathcal R$ in the following steps by replacing the original $\mathcal R$ by a even larger radius if necessary.

Let $u(\cdot,t) \in C(B_\mathcal{R}(-te_n)\to[0,\8))$ and $\bar u:\bar\R^n_+\to P(\R)$ its hodograph transform with respect to $\e>0$ such that
\begin{align*}
\e N \in (0,1) \qquad\text{ where } \qquad N:=\|\bar u\|_{L^\8\1\bar Q_{R}^+\2}+1.
\end{align*}
This implies that for $t\in(R,0]$ and $x\in B_R(-te_n)$
\begin{align}\label{eq:flatness_envelope}
(x_n+t-\e N)_+ \leq u(x,t) \leq (x_n+t+\e N)_+.
\end{align}
(To be precise, $\e N \in (0,1)$ implies the previous inequalities with a radius smaller that $R$, for instance $(R-1)$. This is the type of technicalities that we talk about in the first paragraph and omit in subsequent steps).

In this setting we construct for $\xi, \t \in(0,1)$:
\begin{align*}
u^{\xi, \t}(x,t) &:= \sup_{\substack{t\in\R\\y \in P^{\xi, \t}(s)}}u^*(x+y,t+s),\\
u_{\xi, \t}(x,t) &:= \inf_{\substack{t\in\R\\y \in P_{\xi, \t}(s)}}u_*(x+y,t+s),
\end{align*}
where,
\begin{align*}
P^{\xi, \t}(s) &:= \left\{(y',y_n)\in\R^n:y_n \leq - 2\e N\1\frac{1}{\xi}|y'|^2+\frac{1}{\t}s^2\2-s\right\},\\
P_{\xi, \t}(s) &:= \left\{(y',y_n)\in\R^n:y_n \geq 2\e N\1\frac{1}{\xi}|y'|^2+\frac{1}{\t}s^2\2-s\right\}.
\end{align*}

Recall that $u^*$ and $u_*$ denote the upper and lower semicontinuous envelopes of $u$ defined in \eqref{eq:semi_cont_env}. The main idea behind the construction of $u^{\xi, \t}$ is that for each level set of $u$ we are taking a type of sup-convolution of by touching them with paraboloids from the zero set. See Figure \ref{fig:convolution}.
\begin{figure}[t!]
\begin{center}
\includegraphics[width=12cm]{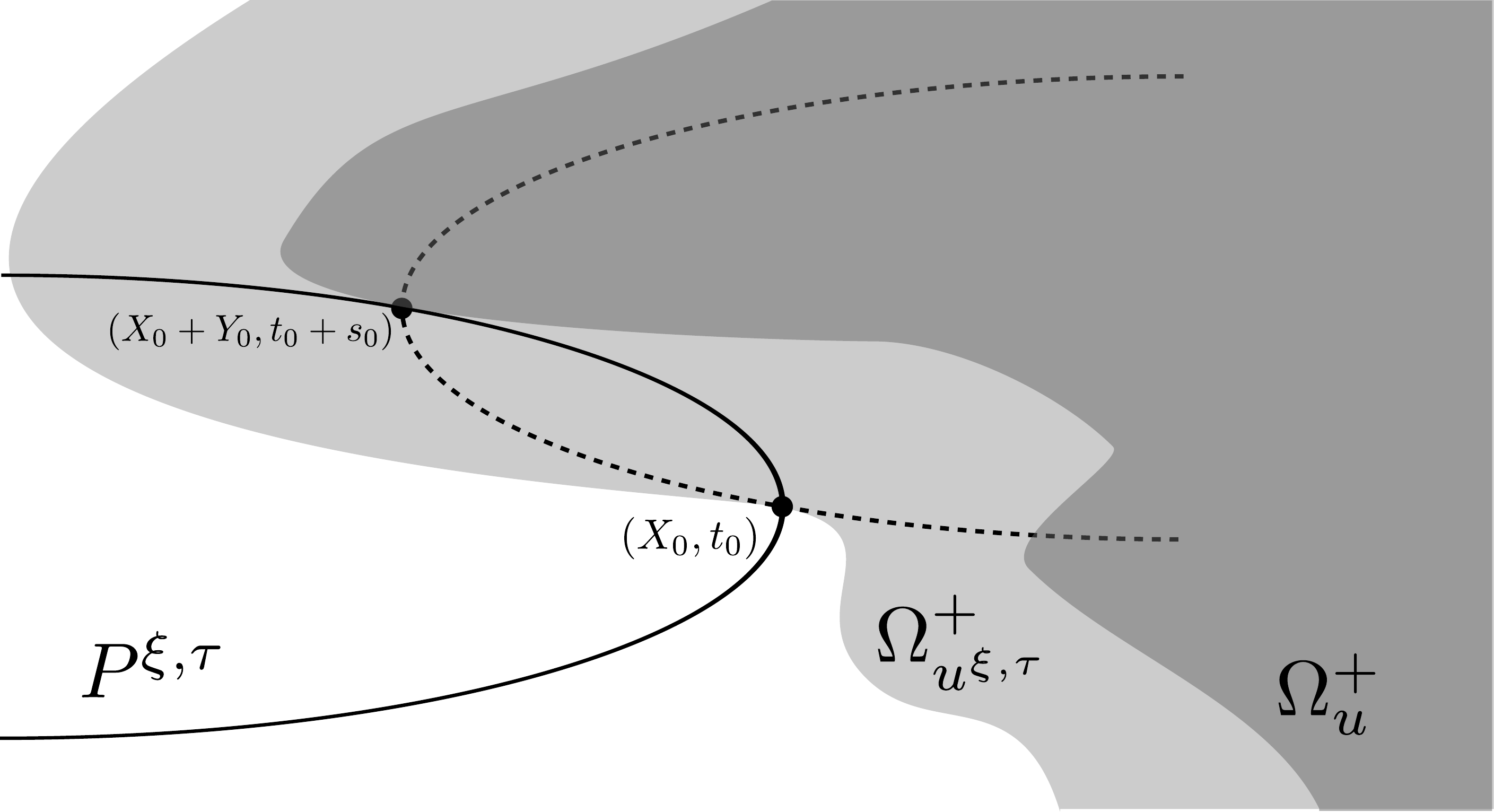}
\end{center}
\caption{Sup convolution of the free boundary.}
\label{fig:convolution}
\end{figure}

\begin{lemma}\label{lem:convolution} Under the flatness hypothesis \eqref{eq:flatness_envelope} we get that:
\begin{enumerate}[label=(\alph*)]
\item\label{flat} \textbf{Flatness:} For $t\in(-R,0]$ and $x\in B_R(-te_n)$,
\[
(x_n+t-\e N)_+ \leq u^{\xi,\t}(x,t) \leq (x_n+t+\e N)_+.
\]
\item\label{dual_pt} \textbf{Dual point:} For $t_0\in(-R,0]$ and $x_0 \in \W^+_{u^{\xi, \t}}(t_0) \cap B_{R}(-t_0e_n)$ let $(y_0,s_0) = (y_0',(y_0)_n, s_0)$ be a point where the supremum in the construction of $u^{\xi,\t}(x_0,t_0)$ is realized. Then 
\[
|s_0|<\t^{1/2}, \qquad |y_0'| < \xi^{1/2}, \qquad |(y_0)_n| < 2\e N, \qquad y_0 \in \p P^{\xi, \t}\1s_0\2,
\]
and
\[
u^{\xi, \t}(\cdot-y_0,\cdot-s_0) \text{ touches $u^*$ from above at $(x_0+y_0,t_0+s_0)$.}
\]
In particular, by considering the limiting case as $x_k \to x_0 \in \G_{u^{\xi,\t}}(t_0)$ we recover the same property for $x_0 \in \G_{u^{\xi, \t}}(t_0)$ with $(x_0+y_0) \in \G_{u^*}(t_0+s_0)$.
\item\label{eq_conv} \textbf{Equations:} If $u=u^*$ is a viscosity subsolution of the Hele-Shaw problem then $u^{\xi, \t}$ is also a viscosity subsolution in $B_{R}(-te_n)$.
\item\label{hod_conv} \textbf{Hodograph:} The hodograph $\bar u^{\xi, \t}$ is single valued in $Q_{R}^+$. Moreover, for every $x_n\in[0,R)$, $\bar u^{\xi, \t}(\cdot,x_n,\cdot)$ can be also computed as a sup-convolution of $\bar u^*(\cdot,x_n,\cdot)$,
\begin{align}\label{eq:hodograph_conv}
\bar u^{\xi,\t}(x',x_n,t) = \sup_{\substack{|s| <\t^{1/2}\\ |y'|<\xi^{1/2}}} \bar u(x'+y',x_n,t+s) - 2N\1\frac{1}{\xi}|y'|^2+\frac{1}{\t}s^2\2.
\end{align}
\item\label{regularity_sup_conv} \textbf{$C^{1,1}$ Regularity from below:} For $(x_0',(x_0)_n,t_0) \in Q_{R}^+$, let $(y_0',s_0)$ a point where the supremum in \eqref{eq:hodograph_conv} is realized, then we have that for all $(x',t) \in Q_{R}^{n-1}$,
\[
\bar u(x_0'+y_0',(x_0)_n,t_0+s_0) - 2N\1\frac{1}{\xi}|x'-(x_0'+y_0')|^2 + \frac{1}{\t}(t-(t_0+s_0))^2\2 \leq \bar u^{\xi,\t}(x',(x_0)_n,t).
\]
\item\label{lipschitz_conv}\textbf{Lipschitz regularity:} The hodograph $\bar u^{\xi,\t}(\cdot,0,t)$ is Lipschitz continuous
\begin{align*}
\sup_{t\in(-R,0]}[\bar u^{\xi,\t}(\cdot,0,t)]_{C^{0,1}\1B_{R}^{n-1}\2} \leq \frac{2N}{\xi^{1/2}} 
\end{align*}
In particular, for every $t\in(-R,0]$, $\W^+_{u^{\xi,\t}}(t)\cap B_{R}(-te_n)$ is a Lipschitz domain.
\item\label{rate}\textbf{Rate of convergence:} For $(x',x_n,t)\in \bar Q_R^+$,
\[
\bar u(x',x_n,t) \leq \bar u^{\xi,\t}(x',x_n,t) \leq \sup_{\substack{ |s|<\t^{1/2}\\|y'| < \xi^{1/2}}} \bar u(x'+y',x_n,t+s).
\]
As $\t\to0$, $\bar u^{\xi,\t}$ decreases pointwise to $\bar u^{\xi,0}$ where,
\[
\bar u(x',x_n,t) \leq u^{\xi,0}(x',x_n,t) := \sup_{\substack{ |y'| < \xi^{1/2}}} \bar u(x'+y',x_n,t) - \frac{2N}{\xi}|y'|^2 \leq \sup_{\substack{ |y'| < \xi^{1/2}}} \bar u(x'+y',x_n,t)
\]
As a function in space and time, $\bar u^{\xi,0}$ is still upper semicontinuous.
\end{enumerate}
\end{lemma}

\begin{remark}
The construction of $\bar u$ using approximating sequences includes its own semicontinuous envelopes. Therefore, the supremum in \eqref{eq:hodograph_conv} is achieved.
\end{remark}

\begin{proof}
Given that $P^{\xi,\t}(s) \ss \{y_n\leq -(\e N/\t)s^2-s\}$ we get the following bounds for $u^{\xi,\t}$ coming from the flatness hypothesis for $u$,
\begin{align}\label{eq:flat}
\1x_n-\frac{2\e N}{\t}s^2 +t - \e N\2_+ &\leq 
\sup_{y\in P^{\xi,\t}(s)} u^*(x+y,t+s),\\
\nonumber &\leq \1x_n-\frac{2\e N}{\t}s^2+t+ \e N\2_+.
\end{align}
This immediately implies both bounds for part \ref{flat}.

For part \ref{dual_pt} we notice that if $|t|\geq \t^{1/2}$ implies that,
\begin{align*}
\1P^{\xi,\t}(t)+x_0\2\cap B_R(-(t_0+t)e_n) &\ss \{x_n < (x_0)_n-t-2\e N\}\cap B_R(-(t_0+t)e_n)\\
&\ss \{x\in B_R(-(t_0+t)e_n):u^{\xi,\t}(x,t_0+t)<u^{\xi,\t}(x_0,t_0)\}.
\end{align*}
The last inclusion follows from the flatness already proven for $u^{\xi,\t}$. Therefore $u^{\xi,\t}(x+x_0,t_0+t)<u^{\xi,\t}(x_0,t_0)$ for $|t|\geq \t^{1/2}$ and $x\in P^{\xi,\t}(t)$ implies that $|s_0|<\t^{1/2}$. From the flatness we also get that,
\begin{align*}
x_0+y_0 &\in (P^{\xi,\t}(s_0)+x_0) \cap \{x\in B_R(-(t_0+s_0)e_n):u^{\xi,\t}(x,t_0+s_0)\geq u^{\xi,\t}(x_0,t_0)\},\\
&\ss (P^{\xi,\t}(s_0)+x_0) \cap \{x_n \geq (x_0)_n-s_0-2\e N\},\\
&\ss \{(z',z_n) \in \R^n: |z'-x_0'| < \xi^{1/2}, \ |z_n-(x_0)_n|<2\e N\}
\end{align*}
The fact that $y_0\in\p P^{\xi,\t}(s_0)$ follows from the maximum principle applied to the harmonic function $u(\cdot,s_0)$. Finally, the inequality $u^{\xi,\t}(x-y_0,t-s_0) \geq u(x,t)$ follows from the construction because,
\begin{align*}
u^{\xi,\t}(x-y_0,t-s_0) = \sup_{\substack{s\in\R\\z\in P^{\xi,\t}(s)}} u^*(x-y_0+y,t-s_0+s) \underset{\substack{s=s_0\\y=y_0\in P^{\xi,\t}(s_0)}}{\geq} u^*(x,t).
\end{align*}

For part \ref{eq_conv} we get that the function $u^{\xi,\t}(\cdot,t)$ is subharmonic because it is the supremum of subharmonic functions. The free boundary equation follows as a consequence of part \ref{dual_pt}. A test function $\varphi$ that touches $u^{\xi,\t}$ from above at $x_0 \in \G_{u^{\xi,\t}}(t_0)$ can be translated by $(y_0,s_0)$ to a get a test function that touches $u^* = u$ from above at $(x_0+y_0) \in \G_{u}(t_0+s_0)$.

For part \ref{hod_conv}, we see that $\bar u^{\xi,\t}$ is single valued if and only if $u^{\xi,\t}(\cdot,t)$ is strictly increasing in the $e_n$ direction, which follows by the strict maximum principle. We prove now that $\bar u^{\xi,\t}(x,x_n,t)$, originally defined as the hodograph of $u^{\xi,\t}$ evaluated at $(x,t)= (x',x_n,t)\in \bar Q^+_R$, can be computed also as a sup-convolution of $\bar u(\cdot,c,\cdot)$ at $(x',t)$. Let,
\[
X = x-(t+\e\bar u^{\xi,\t}(x,t))e_n.
\] 
We obtain that,
\begin{align*}
(P^{\xi,\t}(s)+X)\cap B_R(-(t+s)e_n) &\ss \{Y \in B_{R}(-(t_0+s)e_n):u^*(Y,t+s)\leq u^{\xi,\t}(X,t) = x_n\},\\
&\ss \{Y \in B_{R}(-(t+s)e_n):u(Y,t+s)\leq x_n\}
\end{align*}
In terms of the hodograph it implies that for $y' \in \R^{n-1}$,
\[
\bar u^{\xi,\t}(x,t) +2N\1\frac{1}{\xi}|y'|^2+\frac{1}{\t}s^2\2 \geq \bar u(x'+y',c,t+s).
\]
Hence,
\[
\bar u^{\xi,\t}(x,t) \geq \sup_{s,y'} \bar u(x'+y',c,t+s) -2N\1\frac{1}{\xi}|y'|^2+\frac{1}{\t}s^2\2.
\]
Part \ref{dual_pt} provides the desired equality and the bounds for $s$ and $y'$.

The remaining properties are now well known results for the sup-convolution and can actually be proven using arguments similar to the ones we have already given so far. See for instance, Chapter 5 in \cite{Caffarelli95}.
\end{proof}

The next step is to consider for each $t\in(-R,0]$, the harmonic replacement of the subharmonic function $u^{\xi,\t}(\cdot,t)$ in $\W_{u^{\xi,\t}}^+(t)\cap B_{R}(-te_n)$. The following Dirichlet problem can be solved in the classical sense given that $\W^+_{u^{\xi,\t}}(t)\cap B_{R}(-te_n)$ is a Lipschitz domain,
\begin{alignat*}{3}
\D U^{\xi,\t}(\cdot,t) &= 0 \qquad&&\text{ in }\qquad &&\W_{u^{\xi,\t}}^+(t)\cap B_R(-te_n),\\
U^{\xi,\t}(\cdot,t) &= u^{\xi,\t}(\cdot,t) \qquad&&\text{ on }\qquad &&\p\1\W_{u^{\xi,\t}}^+(t)\cap B_R(-te_n)\2,
\end{alignat*}

Notice that $U^{\xi,\t}(\cdot,t)$ and $\bar U^{\xi,\t}(\cdot,t)$ inherit some of the main characteristics of $u^{\xi,\t}(\cdot,t)$ and $\bar u^{\xi,\t}(\cdot,t)$, namely:
\begin{enumerate}[label=(\alph*)]
\item The flatness property \ref{flat} remains the same for $U^{\xi,\t}(\cdot,t)$.
\item $\bar U^{\xi,\t}(\cdot,t)$ is single valued. This follows from the monotonicity of the boundary data in the $e_n$ direction.
\item $\bar U^{\xi,\t}(\cdot,0,\cdot) = \bar u^{\xi,\t}(\cdot,0,\cdot)$ is $C^{1,1}$ regular from below.
\item The rate of convergence remains the same thanks to the comparison principle.
\end{enumerate}

The advantage of having now a harmonic function is that it rules out any possible degenerate growth of the function at regular points of the free boundary which we define next.

Let $U:\W\to[0,\8)$ be a harmonic function in $\W^+_U$. We say that $X_0 \in \G_U$ is a regular point from the positive set if there exists a ball
\begin{align*}
&B_{r^+}(P^+)\cap \W \ss \W_U^+\text{ such that } X_0 \in \bar B_{r^+}(P^+) \cap \G_U.
\end{align*}
It is well known that in this case $U$ has a precise linear asymptotic behavior along nontangential regions of $\W_U^+$ about $X_0$, see \cite[Section 11.6]{salsa}. In other words, there exists a number $|D^{NT}U(X_0)| \in (0,\8]$ such that for $X \in \W_U^+$ approaching $X_0$ non tangentially,
\begin{align*}
U(X) =  |D^{NT}U(X_0)|(X-X_0)\cdot \nu + o(|X-X_0|),\qquad\text{where}\qquad \nu := \frac{P^+-X_0}{r^+}.
\end{align*}
Moreover, by adding a hypothesis of the form 
\[
U \geq (x_n-((P^+)_n-d))_+ \text{ where $d\in[r^+/2,r^+]$}
\]
we get by a standard barrier argument that for some universal constant $\l>0$,
\[
|D^{NT}U(X_0)|\geq \l.
\]

For $U^{\xi,\tau}(\cdot,t_0)$ we get that every point $X_0 \in \G_{U^{\xi,\tau}}(t_0)$ is regular from the positive set with respect to a ball of radius $r_+ \sim \xi/(2N\e)$. Therefore, by making $\xi = \e$ and assuming that $N^2\e$ is sufficiently small we recover a universal bound from below for $|D^{NT}U^{\xi,\t}|$.

Similar conclusions hold if we assume that $X_0 \in \G_U$ is a regular point from the zero set. In other words, there exists a ball
\begin{align*}
&B_{r^0}(P^0)\cap \W \ss \W_U^0\text{ such that } X_0 \in \bar B_{r^0}(P^0) \cap \G_U.
\end{align*}
Once again there exists a number $|D^{NT}U(X_0)|\in [0,\8)$ such that for $X \in \W_U^+$ approaching $X_0$ non tangentially,
\begin{align*}
U(X) =  |D^{NT}U(X_0)|(X-X_0)\cdot \nu + o(|X-X_0|),\qquad\text{where}\qquad \nu := \frac{X_0-P^0}{r^0}.
\end{align*}
Moreover, for some universal constant $\L>0$
\[
U\leq (x_n-((P^0)_n+d))_+ \text{ where $d\in[r^0/2,r^0]$} \qquad \Rightarrow\qquad |D^{NT}U(X_0)|\leq \L.
\]
We should keep this in mind for $U_{\xi,\tau}(\cdot,t)$ defined as the harmonic replacement of $u_{\xi,\tau}(\cdot,t)$ in $B_R(-te_n)$.

At this point we have we have that we can evaluate at least the spatial ingredient from the free boundary relation for $U^{\xi,\t}$. Notice also that the free boundary gets parametrized by,
\[
\G_{U^{\xi,\t}}(t) = \{(x',x_n,t) \in B_R(-te_n): x_n = -t-\e\bar U^{\xi,\t}(x',0,t)\}.
\]

Recall from Definition \ref{def:gamma_reg} that $x_0 = (x_0',(x_0)_n) \in \G_{U^{\xi,\tau}}(t_0)$ is a regular point in space time if $\bar U^{\xi,\t}$ punctually $C^{1,1}$ at $(x_0',0,t_0)$. In other words, there exist $D_{\R^{n-1}}\bar U^{\xi,\t}(x_0',0,t_0)\in\R^{n-1}$ and $\p_t\bar U^{\xi,\t}(x_0',0,t_0)\in \R$ such that,
\begin{align*}
\bar U^{\xi,\t}(x',0,t) &= P(x',t) + O(|x'-x_0'|^2+|t-t_0|^2),\\
P(x',t) &:= \bar U^{\xi,\t}(x'_0,t_0)+D_{\R^{n-1}}\bar U^{\xi,\t}(x_0',0,t_0)\cdot(x'-x_0')+\p_t\bar U^{\xi,\t}(x_0',0,t_0)(t-t_0).
\end{align*}
Recall that in Section \ref{sec:prelim} we defined the speed of the interphase whenever this can be parametrized by a function which is punctually first order differentiable at a given point.

\begin{lemma}\label{lem:pt_wise_eval}
If $x_0 = (x_0',(x_0)_n) \in \G_{U^{\xi,\tau}}(t_0)$ is a regular point in space time then the free boundary relation can be evaluated in the following classical sense,
\[
\frac{\p_t U^{\xi,\t}}{|DU^{\xi,\t}|}(x_0,t_0) \leq |D^{NT}U^{\xi,\t}(x_0,t_0)|.
\]
\end{lemma}

\begin{proof}
Let $\eta>0$ be a small parameter that will be send to zero at the end of the proof and consider the set,
\[
D_{B,\eta}(t) = \{x_n\leq -t - \e(P(x',t)-B(|x'-x_0'|^2+|t-t_0|^2))\}.
\]
For $B>0$ sufficiently large and $\eta,r>0$ sufficiently small we get that,
\[
D_{B,\eta}(t) \cap B_r(-te_n) \ss \W_{U^{\xi,\tau}}^0(t) \text{ for all $t\in(t_0-r,t_0]$}.
\]
Let $d(\cdot,t)$ be the distance function to $D_{2B,\eta}$ and,
\[
\varphi(x,t) := (|D^{NT}U^{\xi,\t}(x_0,t_0)|+\eta)d(x,t)-Cd(x,t)^2.
\]
Where $C>0$ is a sufficiently large constant such that
\[
\D \varphi(\cdot,t) \leq 0 \text{ in }  B_r(-te_n) \sm D_{2B,\eta}(t).
\]
By making $r$ even smaller if necessary we get that $\varphi$ touches $U^{\xi,\t}$ from above at $(x_0,t_0)$, therefore,
\[
\frac{1+\e\p_t\bar U^{\xi,\t}(x_0',0,t_0)}{\sqrt{1+\e^2|D_{\R^{n-1}}\bar U^{\xi,\t}(x_0',0,t_0)|^2}} =\frac{\p_t\varphi}{|D\varphi|}(x_0,t_0) \leq |D\varphi(x_0,t_0)| = |D^{NT}U^{\xi,\t}(x_0,t_0)|+\eta.
\]
We recover then the desired inequality after sending $\eta$ to zero.
\end{proof}

\subsection{Proof of Theorem \ref{thm:lim_eq}}

Let $\mathcal R_k\to\8$ as in Theorem \ref{thm:lim_eq} and $\mathcal R_k > R_k \to \8$ such that we have enough room to apply the constructions and results in the previous part with respect to $u_k$. Here we use the following notation,
\begin{alignat*}{3}
N_k &:= \|\bar u_k\|_{L^\8\1\bar Q_{R_k}^+\2}+1,\\
V^{k,\t}(x,t) &:= U_k^{\e_k,\tau}\1x+(h/2)e,t-\t^{1/2}\2, \qquad &&v^{k,\t} := \bar V^{k,\t}\\
V_{k,\t}(x,t) &:= U_{k,\e_k,\tau}\1x-(h/2)e,t-\t^{1/2}\2, \qquad &&v_{k,\t} := \bar V_{k,\t}\\
w^{k,\t} &:= v^{k,\t}- v_{k,\t}.
\end{alignat*}

The subsolution inequality of Theorem \ref{thm:lim_eq} becomes now a consequence of the following stability result. The supersolution inequality can be obtained with a similar argument from where we finally settle the proof of Theorem \ref{thm:lim_eq}.

\begin{lemma}
Under the hypothesis of Theorem \ref{thm:lim_eq} there exist a constant $0<\l\leq \L<\8$ such that $w$ satisfies in the viscosity sense
\[
\p_t w \leq \sup_{a\in[\l,\L]} a\p_n w \qquad \text{ in } \qquad \R^{n-1}\times(-\8,0].
\]
\end{lemma}

\begin{proof}
Let us assume without loss of generality that $\varphi \in C^\8(Q_{r}^+)$ is a test function touching $w$ from above at $(0,0)$ such that,
\begin{align*}
&\varphi(x,t) = p \cdot x + s t + O(|x|^2+t^2),\qquad p := D\varphi(0,0)\in \R^n, \qquad s := \p_t \varphi(0,0)\in\R.
\end{align*}
Our goal is to establish that,
\begin{align*}
&s \leq \sup_{a\in[\l,\L]} ap_n.
\end{align*}

Let $B>1$, $\eta\in(0,1)$, and
\begin{align*}
P(x,t) &:= \underbrace{B(|x'|^2 - (n-1)x_n^2)}_{:= Q(x)}+\underbrace{(p+\eta e_n)\cdot x + (s-\eta)t}_{:= L(x,t)}.
\end{align*}
By taking $B$ sufficiently large and then replacing $r$ by a sufficiently small radius,
\begin{align*}
w \leq \varphi < P-(|x|^2+t^2) \text{ in } \bar Q_{r}^+.
\end{align*}
Therefore,
\begin{align*}
\inf_{\bar Q_{\eta r}^+ \sm \bar Q_{\eta r/2}^+} (P - w) &> (\eta r/2)^2 \qquad\text{and}\qquad \inf_{\bar Q_{\eta r/2}^+} (P - w) = 0.
\end{align*}
By the (local) uniform convergence of $\d_{he}\bar u_k\to w$ we obtain that for $k$ sufficiently large,
\begin{align*}
\inf_{\bar Q_{\eta r}^+ \sm \bar Q_{\eta r/2}^+} (P - \d_{he}\bar u_k) &> (1-\s)(\eta r/2)^2 \qquad\text{and}\qquad \inf_{\bar Q_{\eta r/2}^+} (P - \d_{he}\bar u_k) < \s (\eta r/2)^2.
\end{align*}
The constant $\s>0$ is universal and will be determined latter on. Our next step is to choose $\t_k>0$ sufficiently small such that similar relations hold for $w^{k,\t_k}$.

From the H\"older estimate for $\bar u_k$ we already know that,
\[
[\bar u_k]^*_{C^\a_{trun(C\e_k)}\1B_{R_k}^d\2} \leq CN_k.
\]
We consider $k$ large enough such that $CN_k\e_k^{\a/2} < \s(\eta r/2)^2$. By the rate of convergence given by part \ref{rate} in Lemma \ref{lem:convolution} we get now that,
\begin{align*}
\inf_{\bar Q_{3\eta r/4}^+ \sm \bar Q_{\eta r/2}^+} \1P - w^{k,0}\2 &> (1-3\s)(\eta r/2)^2 \qquad\text{and}\qquad \inf_{\bar Q_{\eta r/2}^+} \1P - w^{k,0}\2 < 3\s (\eta r/2)^2,
\end{align*}
At this moment we argue that there exists $\t_k>0$ sufficiently small such that,
\begin{align}\label{eq:wk}
\inf_{\bar Q_{3\eta r/4}^+ \sm \bar Q_{\eta r/2}^+} \1P - w^{k,\t_k}\2 &> (1-5\s)(\eta r/2)^2 \qquad\text{and}\qquad \inf_{\bar Q_{\eta r/2}^+} \1P - w^{k,\t_k}\2 < 5\s (\eta r/2)^2,
\end{align}
Actually for the second relation it suffices to use only the pointwise convergence as $\t\to0$ about a point where the infimum is realized.

The separation over $\bar Q_{3\eta r/4}^+ \sm \bar Q_{\eta r/2}^+$ follows by compactness. Assume that for some sequence $\t_m\to0$ there exists $(x_m,t_m)\to(x_\8,t_\8) \in \bar Q_{\eta r}^+ \sm \bar Q_{\eta r/2}^+$ such that
\[
(P - w^{k,\t_m})(x_m,t_m) \leq (1-5\s)r^2/4.
\]
This implies that there exist $t_m^+,t_m^- \in (t_m-2\t_m^{1/2},t_m)$ such that,
\[
P(x_m,t_m^-) -\1v^{k,0}(x_m,t_m^+) - v_{k,0}(x_m,t_m^-)\2 \leq (1-3\s)r^2/4.
\]
The contradiction now follows because $v^{k,0}$ and $v_{k,0}$ are respectively upper and lower semicontinuous.

Now that $\t_k>0$ has been fixed we denote $V^k := V^{k,\t_k}$, $V_k := V_{k,\t_k}$, $v^k := v^{k,\t_k}$, $v_k := v_{k,\t_k}$ and $w^k := w^{k,\t_k}$. The next step is to consider conformal deformations of $V^k$ and $V_k$ using $P = Q+L$ in order to obtain a contact point between two free boundaries.

Let $\r_k := (\e_k B)^{-1}$ and consider the following Kelvin transform applied to $V^k$,
\begin{align*}
\widetilde V^k(x,t) &:= \1\frac{\r_k}{|x+\r_k e_n|}\2^{n-2}V^k\1\r^2\frac{(x-2x_ne_n)-\r_k e_n}{|x+\r_k e_n|^2}+\r_k e_n,t\2.
\end{align*}
We get that $\widetilde V^k$ is harmonic in its positivity set. From Lemma \ref{lem:kelvin} we know that if $\widetilde v^k$ is the corresponding hodograph then for $k$ sufficiently large,
\begin{align}
\label{eq:q7}\|\widetilde v^k - (v^k - Q)\|_{L^\8\1\bar Q_{\eta r}^+\2} \leq \s(\eta r/2)^2
\end{align}

On the other hand, we consider the conformal deformations of $V_k$ given by,
\begin{align*}
\widetilde V_k\1e^{-\e_k M}x-\e_k(s-\eta)te_n,t\2 = V_k(x,t), \qquad M = (p_n+\eta)Id + e_n\otimes p'-p'\otimes e_n\in \R^{n\times n}.
\end{align*}
Again $\widetilde V_k$ is harmonic in its positivity set because of the relation,
\[
e^{-\e_k M}\1e^{-\e_k M}\2^T = e^{-2\e_k(p_n+\eta)}Id.
\]
From Lemma \ref{lem:rotation} we know that if $\widetilde v_k$ is the corresponding hodograph then,
\begin{align}
\label{eq:q6}\|\widetilde v_k-(v^k + L)\|_{L^\8\1\bar Q_{\eta r}^+\2} \leq \s(\eta r/2)^2
\end{align}

The point is that by substituting $(v_k+L)$ and $(v^k-Q)$ by $\widetilde v_k$ and $\widetilde v^k$ in \eqref{eq:wk} we get that,
\begin{align*}
\inf_{\bar Q_{3\eta r/4}^+ \sm \bar Q_{\eta r/2}^+} (\widetilde v_k - \widetilde v^k) &> (1-7\s)(\eta r/2)^2 \qquad\text{and}\qquad \inf_{\bar Q_{\eta r/2}^+} \1\widetilde v_k - \widetilde v^k\2 < 7\s (\eta r/2)^2.
\end{align*}
At this point we finally declare $\s=1/14$ such that $(1-7\s) = 7\s$. In terms of $\widetilde V_k$ and $\widetilde V^k$ we have that after a suitable translation both graphs come into contact at a common free boundary point, moreover we will enforce this contact to happen at a negative time. To be precise, let $\eta_k>0$ sufficiently small such that,
\[
\theta_k := \inf_{\bar Q_{\eta r/2}^+} \1\widetilde v_k - \widetilde v^k-\eta_kt^{-1}\2 <  7\s (\eta r/2)^2
\]
Notice that if the infimum above is realized at time $t_k$ then from the bounds we have for $v^k$ and $v_k$ get get that
\[
t_k < -C\eta_k N_k^{-1} < 0,
\]
where $C>0$ depends also on $P$ and $B$. Now we set up the translation,
\begin{align*}
\widehat V^k(x,t) &:= \widetilde V^k(x+\e_k(\theta_k + \eta_kt^{-1})e_n,t),
\end{align*}
such that
\begin{align*}
&\widehat V^k(x,t) \leq \widetilde V_k(x,t) \text{ for all $t\in(-\eta r/2,0]$ and $x \in B_{5\eta r/8}(-te_n)\cap \supp \widehat V^k(\cdot,t)$}.
\end{align*}
By the construction, we know that that both graphs come into contact at some point in $\supp\widehat V^k(\cdot,t)$. This can not happen over the boundary $\p B_{5\eta r/8}(-te_n)\cap \supp \widehat V^k(\cdot,t)$ where the inequality is strict, neither over $B_{5\eta r/8}(-te_n)\cap \W_{\widehat V^k}^+(t)$ where both functions are harmonic. Therefore there exists $X_k \in \G_{\widehat V^k}(t_k)\cap \G_{\widetilde V_k}(t_k) \cap B_{5\eta r/8}(-t_k e_n)$.

The one sided $C^{1,1}$ regularity for $v_k$ and $v^k$ implies that $X_k$ is a regular point in space and time for both free boundaries. Indeed, the for the regularity in time we get to use that $t_k<0$. Moreover, there exists balls such that,
\begin{align*}
&B_{r^0_k}(P^0_k) \ss \W_{\widehat V^k}^0(t_k) \text{ such that } X_k \in \bar B_{r^0_k}(P^0_k) \cap \G_{\widehat V^k}^0(t_k),\\
&B_{r^+_k}(P^+_k) \ss \W_{\widetilde V_k}^0(t_k) \text{ such that } X_k \in \bar B_{r^+_k}(P^+_k) \cap \G_{\widetilde V_k}(t_k).
\end{align*}
Given that we have chosen $\xi_k=\e_k$ as the parameter of the convolutions in space we get that the radii are comparable to one which implies that for some universal $0<\l\leq \L<\8$,
\[
\l\leq |D^{NT}\widehat V^k(X_k,t_k)| \leq |D^{NT}\widetilde V_k(X_k,t_k)| \leq \L.
\]

Now we get to evaluate the free boundary relations at $X_k \in \G_{\widehat V^k}(t_k)\cap \G_{\widetilde V_k}(t_k) \cap B_{5\eta r/8}(-t_k e_n)$ in the sense of Lemma \ref{lem:pt_wise_eval} taking into account the corresponding change of variables, see the discussion at the beginning of Section \ref{sec:domain_deformations}. In the following we denote the interior normal vector to $\W^+_{\widehat V^k}(t_k)$ and also $\W^+_{\widehat V^k}(t_k)$ by
\[
\nu_k := \frac{P^+_k-X_k}{r^+_k} = e_n + O((N_k\e_k)^{1/2}).
\]

For the subsolution we have that,
\begin{align*}
V^k(y,t) &= \1\frac{\r_k}{|y-\r_k e_n|}\2^{n-2} \widehat V^k\1 \Phi_k(y) -\e_k(\theta_k+\eta_k t^{-1})e_n,t\2,\\
\Phi_k(y) &:= \r_k^2\frac{(y-2y_n e_n)+\r_k e_n}{|y-\r_k e_n|^2} -\r_k e_n, \qquad \r_k=(\e_k B)^{-1},\\
D\Phi_k(y) &= \1\frac{\r_k}{|y-\r_k e_n|}\2^{2}R_k(y), \qquad R_k(y) \in SO(n) = Id + O(\e_k).
\end{align*}
Denoting $\Phi^{-1}(X_k+\e_k(\theta_k+\eta_k t_k^{-1})e_n) = Y_k$ we get that,
\begin{align*}
\frac{\p_t V^k}{|DV^k|}(Y_k,t_k) &\leq |D^{NT}V^k(Y_k,t_k)|\\
&= \1\frac{\r_k}{|Y_k-\r_ke_n|}\2^n |D^{NT} \widehat V^k(X_k,t_k)|,\\
&= \1\frac{|X_k+(\e_k(\theta_k+\eta_kt_k^{-1})+\r_k)e_n|}{\r_k}\2^n |D^{NT} \widehat V^k(X_k,t_k)|,\\
&\leq (1+C\eta \e_k) |D^{NT} \widehat V^k(X_k,t_k)|,\\
&\leq (1+C\eta \e_k) |D^{NT}\widetilde V_k(X_k,t_k)|.
\end{align*}
for the second inequality we used that $\r_k = (\e_k B)^{-1}$ and $|\eta_kt_k^{-1}| < CN_k$. Developing the first end of the inequality, using that $\|R_k(y_k)\nu_k - e_n\| \leq O((N_k\e_k)^{1/2})$
\begin{align*}
\frac{\p_t V^k}{|DV^k|}(Y_k,t_k) &= \1\frac{|Y_k-\r_k e_n|}{\r_k}\2^{2}\1\frac{\p_t\widehat V^k}{|D\widehat V^k|}(X_k,t_k)+ \cancel{\e_k\eta_k t_k^{-2} e_n\cdot R_k(y_k)\nu_k}\2,\\
&\geq \1\frac{\r_k}{|X_k+(\e_k\theta_k+\eta_kt_k^{-1}+\r_k)e_n|}\2^{2}\frac{\p_t\widehat V^k}{|D\widehat V^k|}(X_k,t_k),\\
&\geq (1-C\eta  \e_k)\frac{\p_t\widetilde V_k}{|D\widetilde V_k|}(X_k,t_k)
\end{align*}

On the other hand, we also evaluate the supersolution free boundary relation for $\widetilde V^k$,
\begin{align*}
V_k(z,t) = \widetilde V_k\1e^{-\e_k M}z-\e_k(s-\eta)te_n,t\2
\end{align*}
Denoting $e^{\e_k M}Z_k -\e_k(s-\eta)t_k = X_k$ we get that,
\begin{align*}
e^{-\e_k(p_n+\eta)}|D^{NT}\widetilde V_k|(X_k,t_k) &= |D^{NT}V_k|(Z_k,t_k),\\
&\leq \frac{\p_t V_k}{|D V_k|}(Z_k,t_k),\\
&= \frac{\p_t \widetilde V_k}{|D \widetilde V_k|}(X_k,t_k) - \e_k (s-\eta) \1e^{-\e_k(e_n\otimes p' - p'\otimes e_n)}\nu\2_n,\\
&\leq \frac{\p_t \widetilde V_k}{|D \widetilde V_k|}(X_k,t_k) - \e_k (s-\eta)(1-C(N_k\e_k)^{1/2})
\end{align*}

Putting the computations together,
\[
\e_k(1-C(N_k\e_k)^{1/2}) (s-\eta) \leq \11+C\eta  \e_k-e^{-2\e_k(p_n+\eta)}\2|D^{NT}\widetilde V_k(X_k,t_k)|
\]
Using the bounds on the gradient and then sending $\e_k\to0$ we finally get that,
\[
(s-\eta) \leq \sup_{a\in[\l,\L]}a(p_n+\eta).
\]
The proof is now completed after sending $\eta$ to zero.
\end{proof}


\section{Further results}\label{sec:further results}

We would like to point out how the method of proof for Theorem \ref{thm: main} can be adapted to the case where the solution remains close to a planar solution with non-constant velocity (i.e. varying slope	). This can be understood as the analogue of an equation with bounded measurable coefficients.

\begin{corollary}\label{cor:bdd_measurable}
Let $0<\l\leq\L<\8$ and,
\[
a:(-1,0]\to[\l,\L] \text{ continuous, } \qquad  A(t) := \int^0_t a(s)ds.
\]
Let $u(\cdot,t) \in C(B_1(A(t)e_n) \times(-1,0]\to[0,\8)$ be a viscosity solution of the Hele-Shaw problem. There exists universal constants $\e_0,\a\in(0,1)$ and $C>0$ depending only on $n$, $\l$ and $\L$, such that if for some $\e\in(0,\e_0)$,
\begin{align*}
a(t)(x_n-A(t)-\e)_+ \leq u(x,t) \leq a(t)(x_n-A(t)+\e)_+,
\end{align*}
then for every $t\in(-1/2,0]$, the free boundary can be parametrized as a $C^{1,\a}$ graph in the $e_n$ direction,
\[
\G_u(t) = \{(x',x_n)\in B_{1/2}: x_n=A(t)-\e\bar u(x',t)\},
\]
with the estimate,
\begin{align*}
\|D_{\R^{n-1}}\bar u\|_{C^\a\1B_{1/2}^{n-1}\times(-1/2,0]\2} \leq C.
\end{align*}
\end{corollary}

By rescaling we can easily obtain the previous consequence however the magnitude of the scaling will depend on the modulus of continuity of $a$ and therefore the resulting estimate would be also dependent on such modulus of continuity. In order to obtain a result which is independent of the modulus of continuity of the slope we could adapt the results of this paper. The strategy is already flexible enough to do this, so let us recapitulate:
\begin{enumerate}[label=(\alph*)]
\item We consider the hodograph in terms of the planar profile $(x,t) \mapsto a(t)(x_n-A(t))_+$,
\[
u(x + (A(t)-\e\bar u(x,t))e_n,t) = a(t)x_n.
\]
\item The diminish of oscillation Lemma \ref{lem:point_est} can be established assuming a flatness hypothesis from the planar profile $(x,t) \mapsto a(t)(x_n-A(t))_+$ with constants depending only on $n$, $\l$ and $\L$. This leads to an estimate for a truncated $\a$-H\"older norm of $\bar u$ as in Theorem \ref{thm:trun_holder}.
\item Section \ref{sec:improvement} can now be reproduced with minor observations concerning the rescalings. Notice that for a rescaling of the form,
\begin{alignat*}{2}
&v(x,t) := \frac{u(\varrho x,\varrho t)}{\varrho}, \qquad &&\bar v(x,t) := \frac{\bar u(\varrho x,\varrho t)}{M},\\
&b(t) := a(\varrho t) \in [\l,\L], \qquad &&B(t) := \int^0_t b(s)ds = \frac{A(\varrho t)}{\varrho}
\end{alignat*}
we get that $v$ is still a viscosity solution of Hele-Shaw and $\bar v$ satisfies,
\begin{align*}
b(t)x_n = v(x+(B(t)-\xi \bar v(x,t))e_n,t),
\end{align*}
where $\xi := \e M/\varrho$. In order words, such rescaling preserves the hypothesis of $\bar v$ being an hodograph of $v$ with respect to a planar profile with slopes in the range $[\l,\L]$.
\item Section \ref{sec:jensen} requires a different construction for the sup/inf convolutions. Here we consider,
\begin{align*}
u^{\xi, \t}(x,t) &:= \sup_{\substack{t\in\R\\y \in P^{\xi, \t}(s)}}u^*(x+y,t+s),\\
u_{\xi, \t}(x,t) &:= \inf_{\substack{t\in\R\\y \in P_{\xi, \t}(s)}}u_*(x+y,t+s),
\end{align*}
where,
\begin{align*}
P^{\xi, \t}(s) &:= \left\{(y',y_n)\in\R^n:y_n \leq - 2\e N\1\frac{1}{\xi}|y'|^2+\frac{1}{\t}s^2\2\right\},\\
P_{\xi, \t}(s) &:= \left\{(y',y_n)\in\R^n:y_n \geq 2\e N\1\frac{1}{\xi}|y'|^2+\frac{1}{\t}s^2\2\right\}.
\end{align*}
Notice that now the paraboloid does not travel with the given planar profile. The reason for this comes from the analogous property \ref{eq_conv} in Lemma \ref{lem:convolution}. If we allow the paraboloids to travel we the planar profile then we would not be able to recover the free boundary condition for $u^{\xi,\t}$ because property \ref{dual_pt} does not necessarily hold. The price we pay is that now for the hodograph we get the following construction,
\begin{align*}
\bar u^{\xi,\t}(x',x_n,t) - \frac{1}{\e}A(t) = \sup_{\substack{|s| <\t^{1/2}\\ |y'|<\xi^{1/2}}} \1 \bar u(x'+y',x_n,t+s) -\frac{1}{\e}A(t+s)\2 - 2N\1\frac{1}{\xi}|y'|^2+\frac{1}{\t}s^2\2.
\end{align*}
Everything works fine in terms of the one sided $C^{1,1}$ regularity in space and time. Moreover, making $\xi=\e$ we still recover the bounds for $|D^{NT}U^{\xi,\t}|$ at the regular points of the free boundary. Notice that this ultimately leads to new ellipticity constants $\l'=c\l$ and $\L' = C\L$ for some $c,C>0$ depending only on the dimension.

In terms of the rate of convergence we have that,
\[
\bar u(x',x_n,t) \leq \bar u^{\xi,\t}(x',x_n,t) \leq \sup_{\substack{ |s|<\t^{1/2}\\|y'| < \xi^{1/2}}} \bar u(x'+y',x_n,t+s) + \frac{2\t^{1/2}\L}{\e}.
\]
We still have a control from the oscillation in space coming from the $\a$-H\"older estimate as we did in the beginning of the proof of Theorem \ref{thm:lim_eq}. On the other hand, we also need to control the oscillation in time which can be achieved for each $k$ by taking $\t_k$ sufficiently small. Notice that now $\t_k \ll \e_k^2$ in order to control the second term in the relation above.
\end{enumerate}

Once we know that the free boundary is at least $C^1$ regular in space we obtain that the free boundary is also analytic in space. This result was proved in \cite{kim2006regularity} by using the transformation given by  M. Elliott and Janovsk{\`y} in \cite{elliott1981variational}. The main observation is that the Hele-Shaw problem can be reformulated as an obstacle problem with $C^1$ free boundary at every given time.

\begin{corollary}
Under the hypothesis of Theorem \ref{thm: main} or Corollary \ref{cor:bdd_measurable} we get that for each $t\in(-1,0]$, $\bar u(\cdot,t)$ is analytic.
\end{corollary}

One of the main interest of having a regularity estimate coming from a local flatness hypothesis is that in many cases these type of hypothesis are actually satisfied for a large set of points. This is one of the driving themes in the theory of minimal surfaces which has been also developed for elliptic and parabolic free boundary problems, see for instance the book \cite{salsa}.


\section{Appendix}\label{sec:appendix}

\subsection{Barriers}

\begin{lemma}\label{lem:appendix}
Given $r\in(0,1)$ consider,
\begin{align*}
p&:=\frac{1}{8r}-\frac{r}{2},\quad\text{ and } \quad R:=\frac{1}{8r}+\frac{r}{2},
\end{align*}
such that $\p B_R(pe_n)$ is the unique sphere that contains the point $-re_n$ and the $(n-2)$ dimensional sphere $\p B_{1/2}^{n-1}$. Let $\e\in(0,1)$ and $u$ satisfies,
\begin{alignat*}{3}
\D u &= 0  \qquad&&\text{ in }  \qquad&&\W \sm B_{1/16}((1/4-\e)e_n),\\
u &= x_n^+ \qquad&&\text{ on } \qquad&&\p\1\W \sm B_{1/16}((1/4-\e)e_n)\2,
\end{alignat*}
where,
\begin{align*}
\W = \1B_R(pe_n)\cap\{x_n\leq 0\}\2 \cup \1B_{3/4}(-\e e_n)\cap\{x_n>0\}\2.
\end{align*}
Then there exists universal constants $r_0,\e_0\in(0,1)$ and $C>0$ such that,
\begin{align*}
r \in (0,r_0) \qquad\text{and}\qquad \e\in(0,\e_0) \qquad\Rightarrow\qquad \p_n u \geq 1-Cr.
\end{align*}
\end{lemma}

\begin{proof}
Given that $u\geq x_n^+$ in $\W$, it implies already that $\p_nu \geq 1$ in $\p\W\cap\{x_n>0\}$. It suffices to show now that the estimate holds for $x \in \p\W\cap\{x_n<0\}$ and conclude by the maximum principle.

Let us consider,
\begin{align*}
U(x) = \frac{1/16-r}{f(R-1/16)-fR}\1 f(|x-pe_n|) - fR\2,
\end{align*}
where,
\begin{align*}
f(\r) := \begin{cases}
-\ln\r &\text{ if  $n=2$},\\
\r^{2-n} &\text{ if $n>2$}.
\end{cases}
\end{align*}
We get then that $U$ and $u$ are both harmonic functions in $\W$ vanishing on $\p B_R(pe_n)\cap\{x_n< 0\}$. Moreover, $U \leq u$ in $(B_R(pe_n)\sm B_{R-1/16}(pe_n)) \cap \Sigma \ss \W$, where $\Sigma$ is a cone with vertex at $pe_n$ and spanned by $B_{1/2}^{n-1}$. Therefore, the problem reduces to getting a similar lower bound for $\p_n U$, which is a straightforward computation,
\[
\p_n U = e_n\cdot\1\frac{pe_n-x}{R}\2 \frac{1/16-r}{f(R-1/16)-fR}f'R = \ldots = 1-Cr+o(r^2).
\]
\end{proof}

\begin{figure}[t]
\begin{center}
\includegraphics[width=10cm]{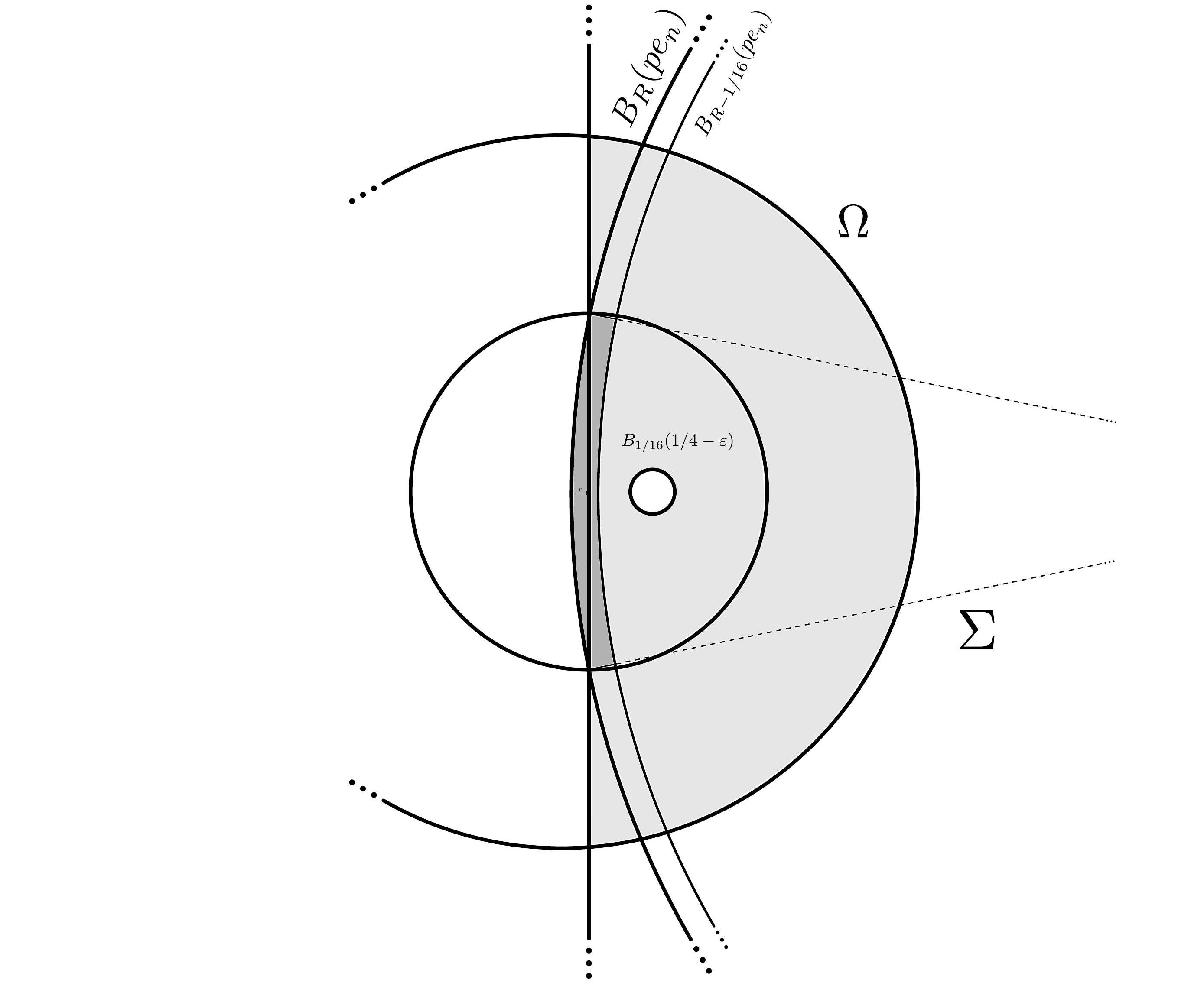}
\end{center}
\caption{Geometric configuration of Lemma \ref{lem:appendix1}}
\label{fig:appendix1}
\end{figure}

\subsection{Interpolation Lemmas}

In this appendix we examine a few interpolation results about truncated H\"older spaces of multi-valued function. The following result is a type of maximum principle.

\begin{lemma}\label{lem:appendix1}
For $h_0\in[0,1)$ and $v:[-1,1]\to P(\R)\sm \emptyset$,
\begin{align*}
&v(\pm 1) \cap (-\8,0] \neq \emptyset, \qquad\text{ and }\qquad \bigcup_{\substack{h\in(h_0,1)\\x\in[-1+h,1-h]}} \d^2_h v(x) \ss [-1,\8),
\end{align*}
implies $v(x) \ss (-\8,1]$ for all $x\in[-1+h_0,1-h_0]$.
\end{lemma}

\begin{proof}
Assume by contradiction and without loss of generality that there exists $x\in [0,1-h_0]$ that realizes the strictly positive maximum of $(v^0-1)$ in $[-1+h_0,1-h_0]$. Then we obtain the following contradiction,
\[
-1 \leq (v^0(2x-1) - v^0(x)) - (v^0(x) - v_0(1)) < 0 - 1.
\]
\end{proof}

The following result is adapted from Lemma 5.6 in \cite{Caffarelli95}. 

\begin{lemma}\label{lem:interpolation}
For $h_0\in[0,1)$, $\a,\b\in(0,1)$ such that $\b+\a<1$, and $v:[-1,1]\to P(\R)\sm \emptyset$, the following estimate holds for some $C>0$ depending only on $(1-(\b+\a))$,
\begin{align*}
\sup_{h\in(h_0,2)} \left\|\frac{\d_{h} v}{h^{\b+\a}}\right\|_{L^\8[-1+h/2,1-h/2]} \leq C\1\osc_{[-1,1]} v + \sup_{h\in(h_0,1)} \left\|\frac{\d_{h}^2 v}{h^{\b+\a}}\right\|_{L^\8[-1+h,1-h]}\2.
\end{align*}
\end{lemma}

\begin{proof}
Consider without loss of generality $H \in (h_0,1/8)$ and $x\in[-1+H/2,0]$. Let $N\geq 2$ be an integer such that $x-H/2+2^NH \leq 1 < x-H/2+2^{N+1}H$, it implies $2^NH \geq 1/2$. Let us consider also,
\[
x_0 = x-H/2, \qquad x_i = x_0+2^{i-1}H, \qquad v_i \in v(x_i) \text{ for $i=0,1,\ldots, N$},
\]
together with the second order differences for $i=1,2,\ldots, (N-1)$,
\[
\qquad \d^2v_i = v_0 - 2v_i +  v_{i+1} \in \d^2_{2^{i-1}H}v(x_i),
\]
such that,
\[
|\d^2v_i| \leq (2^{i-1}H)^{\b+\a}\sup_{h\in(h_0,1)} \left\|\frac{\d_h^2 v}{h^{\b+\a}}\right\|_{L^\8[-1+h,1-h]}.
\]

From the telescopic identity
\[
(2^N-1)v_0 - 2^N v_1 + v_N = \sum_{i=1}^{N-1} 2^{N-i-1}\d^2 v_i
\]
we obtain that,
\[
|v_1-v_0| \leq \underbrace{2^{-N}}_{\leq 2H}\osc_{[-1,1]} v + \frac{1}{1-2^{-(1-(\b+\a))}}H^{\b+\a}\sup_{h\in(h_0,1)} \left\|\frac{\d_h^2 v}{h^{\b+\a}}\right\|_{L^\8[-1+h,1-h]}
\]
therefore the following estimate independently of the choices of $v_0\in v(x_0)$ and $v_1\in v(x_1)$
\[
\frac{|v_1-v_0|}{H^{\b+\a}} \leq C\1\osc_{[-1,1]} v + \sup_{h\in(h_0,1)} \left\|\frac{\d_h^2 v}{h^{\b+\a}}\right\|_{L^\8[-1+h,1-h]}\2.
\]
The conclusion now follows by taking the corresponding supremum on the left-hand side.
\end{proof}

\begin{corollary}\label{lem:appendix3}
For $h_0\in[0,1)$, $h_1\in(h_0,2)$, $\a,\b\in(0,1)$ such that $\b+\a<1$, and $v:[-1,1]\to P(\R)\sm \emptyset$, the following estimate holds for some $C>0$ depending only on $(1-(\b+\a))$,
\[
\sup_{h\in(h_0,h_1)}\left\|\frac{\d_h v}{h^{\b+\a}}\right\|_{L^\8[-1+h/2,1-h/2]} \leq \left\|\frac{\d_{h_1} v}{h_1^{\b+\a}}\right\|_{L^\8[-1+h_1/2,1-h_1/2]} + C\sup_{h\in(h_0,1)} \left\|\frac{\d_h^2 v}{h^{\b+\a}}\right\|_{L^\8[-1+h,1-h]}.
\]
In particular,
\begin{align*}
\sup_{h\in(h_1,2)}\left\|\frac{\d_h v}{h^\b}\right\|_{L^\8[-1+h/2,1-h/2]} \geq \ &\frac{1}{2}\sup_{h\in(h_0,2)}\left\|\frac{\d_h v}{h^\b}\right\|_{L^\8[-1+h/2,1-h/2]}\\
&{} - Ch_1^\a\sup_{h\in(h_0,1)} \left\|\frac{\d_h^2 v}{h^{\b+\a}}\right\|_{L^\8[-1+h,1-h]}.
\end{align*}
\end{corollary}

\begin{proof}
Let $x_0 \in [-1+h_1/2,1-h_1/2]$, $v_- \in v(x_0-h_1/2)$, $v_+ \in v(x_0+h_1/2)$ , and consider,
\[
w(x) = v((h_1/2)x+x_0) - L(x), \qquad L(x) = mx+b, \qquad m=\frac{v_+-v_-}{2}, \qquad b=\frac{v_++v_-}{2}.
\]
The construction is given such that $0 \in w(\pm 1)$ and
\[
\d_h w(x) = \d_{(h_1/2)h} v((h_1/2)x+x_0) - mh, \qquad \d_h^2 w(x) = \d_{(h_1/2)h}^2 v((h_1/2)x+x_0).
\]
Applying the previous lemma to $w$ and using Lemma \ref{lem:appendix1} to bound the oscillation by the second order difference we get
\begin{align*}
\sup_{h\in(2h_0/h_1,2)} \left\|\frac{\d_h w}{h^{\b+\a}}\right\|_{L^\8[-1+h/2,1-h/2]} \leq C\sup_{h\in(2h_0/h_1,1)} \left\|\frac{\d_h^2 w}{h^{\b+\a}}\right\|_{L^\8[-1+h,1-h]}.
\end{align*}
Therefore,
\begin{align*}
\sup_{h\in(h_0,h_1)}\left\|\frac{\d_h v}{h^{\b+\a}}\right\|_{L^\8[x_0+h/2,x_0-h/2]} &\leq \frac{2|m|}{h_1^{\b+\a}} + \1\frac{2}{h_1}\2^{\b+\a}\sup_{h\in(2h_0/h_1,2)} \left\|\frac{\d_h w}{h^{\b+\a}}\right\|_{L^\8[-1+h/2,1-h/2]},\\
&\leq \left\|\frac{\d_{h_1} v}{h_1^{\b+\a}}\right\|_{L^\8[-1+h_1/2,1-h_1/2]} + C\sup_{h\in(h_0,1)} \left\|\frac{\d_h^2 v}{h^{\b+\a}}\right\|_{L^\8[-1+h,1-h]}.
\end{align*}
The first part of the corollary now follows by taking the supremum over $x_0 \in [-1+h_1/2,1-h_1/2]$. For the last part we denote for $a<b$,
\begin{align*}
I_{a,b} &:= \sup_{h\in(a,b)}\left\|\frac{\d_h v}{h^{\b}}\right\|_{L^\8[-1+h/2,1-h/2]}.
\end{align*}
From the first result we get that
\[
I_{h_1,2} + Ch_1^\a\sup_{h\in(h_0,1)} \left\|\frac{\d_h^2 v}{h^{\b+\a}}\right\|_{L^\8[-1+h,1-h]} \geq I_{h_0,h_1}.
\]
By adding $I_{h_0,h_1}+I_{h_1,2} \geq I_{h_0,2}$ we obtain that,
\[
2I_{h_1,2} \geq I_{h_0,2}-Ch_1^\a\sup_{h\in(h_0,1)} \left\|\frac{\d_h^2 v}{h^{\b+\a}}\right\|_{L^\8[-1+h,1-h]},
\]
which is the desired estimate.
\end{proof}

Finally, this last lemma establishes a H\"older estimate for the derivative of single-valued functions when $\a+\b>1$.

\begin{lemma}\label{lem:appendix4}
Let $\a,\b\in(0,1)$ such that $\a+\b>1$ and $v:[-1,1]\to\R$ such that,
\[
\sup_{h\in(0,2)}\left\|\frac{\d_h v}{h^\b}\right\|_{C^\a(-1+(h/2),1-(h/2))} \leq 1.
\]
Then for some universal constant $C$,
\[
\|\partial v\|_{C^{1,\a+\b-1}(-1,1)} \leq C.
\]
\end{lemma}

\begin{proof}
By Lemma 5.6 from \cite{Caffarelli95} we know that $v$ is Lipschitz and therefore differentiable almost everywhere. By a density argument it suffices to show that that for each point of differentiability $x_0\in(-1,1)$,
\begin{align*}
|v(x)-v(x_0)-\partial v(x_0)(x-x_0)| \leq C|x|^{\a+\b} \text{ for } x \in [-1,1].
\end{align*}

Assume without loss of generality that $x_0 = v(x_0) = \partial v(x_0) = 0$. If there exists $h\in(0,2)$ such that $u(h) > Ch^{\a+\b}$, then by iterating the hypothesis of the Lemma we get for every positive integer $i$,
\begin{align*}
\frac{v(2^{-i}h)}{2^{-i}h} > \1C-\sum_{j=0}^{i-1}2^{-(\a+\b-1)j}\2 h^{\a+\b-1} \geq \frac{C}{2}h^{\a+\b-1}>0,
\end{align*}
provided that $C = 4/(2^{\a+\b-1}-1)$. This contradicts $\partial_t u(0)=0$ as $i\to\8$. 
\end{proof}

\subsection{Conformal Deformations}\label{sec:domain_deformations}

We review some facts about conformal deformations used in Section \ref{sec:jensen} in order to determine the blow limit. It is worthwhile to recall conformal mappings have been useful throughout free boundary --in particular, in earlier works dealing with planar domains, where conformal invariance was greatly exploited. More recently, they were used to construct appropriate supersolutions in work of Choi, Jerison, and Kim - these were important in propagating the $\varepsilon$-monotonicity, see \cite[Lemma 3.2]{MR2306045}. 

A smooth bijection $\Phi:\R^n\to\R^n$ is conformal if $D\Phi$ can be factored as a positive scaling times a rotation which we will denote as,
\[
D\Phi(x) = \varphi(x) R(x), \qquad \varphi(x) > 0,\qquad R(x) \in SO(n).
\]
Clearly $\Phi^{-1}$ is also a conformal and,
\[
D\Phi^{-1} = \frac{1}{\varphi(x)} R^{-1}(x) = \frac{1}{\varphi(x)}  R^{T}(x).
\]

Let $\psi:\R^n\to (0,\8)$ smooth. A deformation of a given function $V$ by $\Phi$ and $\psi$ is given by
\[
\widetilde V(y) = \frac{1}{\psi(\Phi^{-1}(y))}V(\Phi^{-1}(y)) \qquad\Leftrightarrow\qquad \psi(x)\widetilde V(\Phi(x)) = V(x).
\]
If $\psi$ is identically equal to one, then the level sets of $\widetilde V$ are given by deforming the level set of $V$ using $\Phi$,
\[
\{\widetilde V \leq c\} = \Phi\{V \leq c\}.
\]
If $V$ is smooth and $V(x_0) = \widetilde V(\Phi(x_0)) = 0$ one gets that,
\[
|D \widetilde V(\Phi(x_0))| = \frac{1}{\psi(x_0)\varphi(x_0)} |D V(x_0)|.
\]

In this work we define $|D^{NT}V(x_0)|$ whenever for some unit vector $\nu\in\p B_1$, $V$ has an asymptotic development about $x_0$ of the form,
\[
V(x) = |D^{NT}V(x_0)|(x-x_0)\cdot\nu + o(|x-x_0|),
\] 
where $x$ approaches $x_0$ in non tangential regions of $\W_V^+$ about $x_0$. Also in this case one easily checks that $\widetilde V$ has an asymptotic development about $\Phi(x_0)$ of the form,
\[
\widetilde V(x) = \frac{1}{\psi(x_0)\varphi(x_0)} |D^{NT}V(x_0)|(x-\Phi(x_0))\cdot R(\nu) + o(|x-\Phi(x_0)|)
\]
Therefore the following change of variables formula also holds under this weaker definition,
\[
|D^{NT} \widetilde V(\Phi(x_0))| = \frac{1}{\psi(x_0)\varphi(x_0)} |D^{NT} V(x_0)|.
\]

Another change of variables we are interested is for the speed of free boundary of $V$ where $V$ is a non negative function depending also on time. Let for $f(t) \in \R^n$ smooth,
\[
\psi(x)\widetilde V(\Phi(x)+f(t),t) = V(x,t).
\]
In the smooth case,
\[
\frac{\p_t V}{|D V|} = \frac{\p_t \widetilde V + \p_t f\cdot D\widetilde V}{\varphi |D \widetilde V|} = \frac{1}{\varphi} \frac{\p_t \widetilde V}{|D \widetilde V|} + \frac{\p_t f}{\varphi}\cdot \frac{D\widetilde V}{|D \widetilde V|}.
\]
A similar formula also holds at a space time regular point $x_0\in \G_V(t_0)$ (Definition \ref{def:gamma_reg})  for which $\nu \in \p B_1$ the interior normal vector to $\W_V^+(t_0)$ is well defined. Then,
\[
\frac{\p_t V}{|D V|}(x_0,t_0) = \frac{1}{\varphi(x_0)} \frac{\p_t \widetilde V}{|D \widetilde V|}(\Phi(x_0)+f(t_0),t_0) + \frac{\p_t f(t_0)}{\varphi(x_0)}\cdot R(\nu).
\]
To prove this claim we notice first that the factor $\psi>0$ is actually irrelevant with respect to the free boundary. Then the formula holds by applying the deformation to a given parametrization of the free boundary and using the standard chance of variable formulas.

Let $\Phi$ be the composition of an inversion about $\p B_\r(\r e_n)$ followed by a reflection about the plane $\{x_n=0\}$,
\begin{align*}
y &= \Phi(x) = \r^2\frac{(x-2x_n e_n)+\r e_n}{|x-\r e_n|^2} - \r e_n = x+\r^{-1}\12x_n x - |x|^2e_n\2 + O(\r^{-2}|x|^2)\\
x &= \Phi^{-1}(y) = \r^2\frac{(y-2y_n e_n)-\r e_n}{|y+\r e_n|^2} + \r e_n = y-\r^{-1}\12y_n y - |y|^2e_n\2 + O(\r^{-2}|y|^2)
\end{align*}
such that,
\[
\frac{|y+\r e_n|}{\r}\frac{|x-\r e_n|}{\r} = 1.
\]
Given $V$ we consider the following type of Kelvin transform,
\begin{align*}
\widetilde V(y) &= \1\frac{\r}{|y+\r e_n|}\2^{n-2}V(\Phi^{-1}(y)), \qquad\qquad V(x) = \1\frac{\r}{|x-\r e_n|}\2^{n-2}\widetilde V(\Phi(x)).
\end{align*}
The idea is that the hodograph of $\widetilde V$ can be estimated by the hodograph of $V$ modulo a quadratic correction.

\begin{lemma}\label{lem:kelvin}
Let $v$ and $\widetilde v$ be the hodographs of $V$ and $\widetilde V$ with respect to $\e\in(0,1)$,
\[
V(x-\e v(x)e_n) = x_n^+ \qquad\qquad \widetilde V(x-\e \widetilde v(x)e_n) = x_n^+.
\]
For any $\s\in(0,1)$ and $\r=(B\e)^{-1}$ we get that,
\[
\|\widetilde v - (v - B(|x'|^2-(n-1)x_n^2))\|_{L^\8(B_1^+)} \leq \s
\]
assuming that $\e^\a N$ is sufficiently small where,
\[
N = \|v\|_{L^\8(B_1^+)}+1 \qquad\text{and}\qquad [v]_{C^\a_{trun(C\e)}(B_1^+)} \leq CN.
\]
\end{lemma}

\begin{proof}
Let,
\begin{align*}
W(x) &:= \widetilde V(\Phi(x-\e v(x)e_n))\\
&= \left|\frac{x-(\e v(x)+\r)e_n}{\r}\right|^{n-2}V(x-\e v(x)e_n)\\
&= \11-\e[B(n-2)x_n+O(\e N)]\2 x_n^+
\end{align*}
Therefore, if $w$ denotes the hodograph of $W$ we get that,
\[
\|w-B(n-2)x_n^2\|_{L^\8(B_1^+)} \leq C\e N.
\]

On the other hand,
\begin{align*}
x_n^+ = W(x-\e w(x)e_n) = \widetilde V(\Phi(x-\e[w(x)+v(x-\e w(x)e_n)]e_n)).
\end{align*}
This allows us to consider,
\begin{align*}
y-\e\widetilde v(y)e_n = \Phi(x-\e[w(x)+v(x-\e w(x)e_n)]e_n) \qquad\text{where}\qquad y=(y',y_n)=(y',x_n).
\end{align*}
For $\e N$ sufficiently small and $x\in B_1^+$ we get that
\[
|x-\e[w(x)+v(x-\e w(x)e_n)]e_n|\leq 2 \qquad\Rightarrow\qquad |y-x| = |y'-x'| \leq C\e,
\]
where the implication follows from the expansion of $\Phi$ in terms of $\r^{-1} = B\e$. Also from the same expansion,
\begin{align*}
y_n-\e\widetilde v(y) &= y_n-\e [w(x)+v(x-\e w(x)e_n)-B(|y'|^2-y_n^2)+O(\e)]\\
&= y_n-\e [v(y)-B(|y'|^2-(n-1)y_n^2)+O(\e^\a N)].
\end{align*}
This implies the desired estimate.
\end{proof}

Let us consider now another type of conformal deformation for $p = (p',p_n)$,
\begin{align*}
V_p(e^{-\e M}x) = V(x),\qquad\text{where}\qquad M = p_nId + e_n\otimes p' - p'\otimes e_n \in \R^{n\times n}.
\end{align*}
In this case the hodograph of $V_p$ can be estimated by the hodograph of $V$ modulo a linear deformation.

\begin{lemma}\label{lem:rotation}
Let $v$ and $v_p$ be the hodographs of $V$ and $V_p$ with respect to $\e\in(0,1)$,
\[
V(x-\e v(x)e_n) = x_n^+, \qquad\qquad V_p(x-\e v_p(x)e_n) = x_n^+.
\]
Then for $\s\in(0,1)$
\[
\|v_p - (v + p\cdot x)\|_{L^\8(B_1^+)} \leq \s.
\]
assuming that $\e^\a N$ is sufficiently small where,
\[
N = \|v\|_{L^\8(B_1^+)}+1 \qquad\text{and}\qquad [v]_{C^\a_{trun(C\e)}(B_1^+)} \leq CN.
\]
\end{lemma}

\begin{proof}
As before we get that,
\begin{align*}
V_p(e^{-\e M}(x-\e v(x)e_n)) = V(x-\e v(x)e_n) = x_n^+
\end{align*}
Letting
\begin{align*}
y-\e v_p(y)e_n = e^{-\e M}(x-\e v(x)e_n) \;\;\textnormal{where}\;\; y=(y',y_n)=(y',x_n).
\end{align*}
we still get that $|x-y|=|x'-y'|\leq C\e$ and,
\begin{align*}
y_n-\e v_p(y) &= y_n -\e\1v(x)+ (M x)_n + O(\e)\2\\
&= y_n -\e\1v(y)+p\cdot y + O(\e^\a N)\2,
\end{align*}
which implies the desired estimate.
\end{proof}

\bibliographystyle{plain}
\bibliography{mybibliography}

\def\cprime{$'$}
\begin{thebibliography}{10}

\bibitem{Allen2015two}
Mark Allen, Erik Lindgren, and Arshak Petrosyan.
\newblock The two-phase fractional obstacle problem.
\newblock {\em SIAM Journal on Mathematical Analysis}, 47(3):1879--1905, 2015.

\bibitem{MR1394964}
I.~Athanasopoulos, L.~Caffarelli, and S.~Salsa.
\newblock Caloric functions in {L}ipschitz domains and the regularity of
  solutions to phase transition problems.
\newblock {\em Ann. of Math. (2)}, 143(3):413--434, 1996.

\bibitem{MR1397563}
I.~Athanasopoulos, L.~Caffarelli, and S.~Salsa.
\newblock Regularity of the free boundary in parabolic phase-transition
  problems.
\newblock {\em Acta Math.}, 176(2):245--282, 1996.

\bibitem{MR1486632}
I.~Athanasopoulos, L.~A. Caffarelli, and S.~Salsa.
\newblock Phase transition problems of parabolic type: flat free boundaries are
  smooth.
\newblock {\em Comm. Pure Appl. Math.}, 51(1):77--112, 1998.

\bibitem{salsa}
Luis Caffarelli and Sandro Salsa.
\newblock {\em A geometric approach to free boundary problems}, volume~68 of
  {\em Graduate Studies in Mathematics}.
\newblock American Mathematical Society, Providence, RI, 2005.

\bibitem{MR990856}
Luis~A. Caffarelli.
\newblock A {H}arnack inequality approach to the regularity of free boundaries.
  {I}. {L}ipschitz free boundaries are {$C^{1,\alpha}$}.
\newblock {\em Rev. Mat. Iberoamericana}, 3(2):139--162, 1987.

\bibitem{MR973745}
Luis~A. Caffarelli.
\newblock A {H}arnack inequality approach to the regularity of free boundaries.
  {II}. {F}lat free boundaries are {L}ipschitz.
\newblock {\em Comm. Pure Appl. Math.}, 42(1):55--78, 1989.

\bibitem{caffarelli1989harnack}
Luis~A Caffarelli.
\newblock A harnack inequality approach to the regularity of free boundaries
  part ii: Flat free boundaries are lipschitz.
\newblock {\em Communications on Pure and Applied Mathematics}, 42(1):55--78,
  1989.

\bibitem{Caffarelli95}
Luis~A. Caffarelli and Xavier Cabr{\'e}.
\newblock {\em Fully nonlinear elliptic equations}, volume~43.
\newblock American Mathematical Society Colloquium Publications, 1995.

\bibitem{Caginalp1989}
G.~Caginalp.
\newblock Stefan and {H}ele-{S}haw type models as asymptotic limits of the
  phase-field equations.
\newblock {\em Phys. Rev. A (3)}, 39(11):5887--5896, 1989.

\bibitem{Caginalp1998ConvergencePhaseField}
Gunduz Caginalp and Xinfu Chen.
\newblock Convergence of the phase field model to its sharp interface limits.
\newblock {\em European J. Appl. Math.}, 9(4):417--445, 1998.

\bibitem{chang2014h}
H.~A. Chang-Lara and G.~D\'{a}vila.
\newblock H{\"o}lder estimates for non-local parabolic equations with critical
  drift.
\newblock {\em Journal of Differential Equations}, 260(5):4237 -- 4284, 2016.

\bibitem{2015arXiv150406294C}
H.~A. {Chang-Lara} and D.~{Kriventsov}.
\newblock {Further Time Regularity for Fully Non-Linear Parabolic Equations}.
\newblock {\em Mathematical Research Letters}, to appear.

\bibitem{2015arXiv150507889C}
H.~A. {Chang-Lara} and D.~{Kriventsov}.
\newblock {Further Time Regularity for Non-Local, Fully Non-Linear Parabolic
  Equations}.
\newblock {\em Communications on Pure and Applied Mathematics}, to appear.

\bibitem{MR3148110}
H{\'e}ctor Chang-Lara and Gonzalo D{\'a}vila.
\newblock Regularity for solutions of non local parabolic equations.
\newblock {\em Calc. Var. Partial Differential Equations}, 49(1-2):139--172,
  2014.

\bibitem{MR2306045}
Sunhi Choi, David Jerison, and Inwon Kim.
\newblock Regularity for the one-phase {H}ele--{S}haw problem from a
  {L}ipschitz initial surface.
\newblock {\em Amer. J. Math.}, 129(2):527--582, 2007.

\bibitem{MR2603767}
Sunhi Choi, David Jerison, and Inwon Kim.
\newblock Local regularization of the one-phase {H}ele--{S}haw flow.
\newblock {\em Indiana Univ. Math. J.}, 58(6):2765--2804, 2009.

\bibitem{DaskalopoulosLee2004AllTimeSmoothSolutions}
P.~Daskalopoulos and Ki-Ahm Lee.
\newblock All time smooth solutions of the one-phase {S}tefan problem and the
  {H}ele-{S}haw flow.
\newblock {\em Comm. Partial Differential Equations}, 29(1-2):71--89, 2004.

\bibitem{MR2813524}
D.~De~Silva.
\newblock Free boundary regularity for a problem with right hand side.
\newblock {\em Interfaces Free Bound.}, 13(2):223--238, 2011.

\bibitem{elliott1981variational}
Charles~M Elliott and Vladim{\i}r Janovsk{\`y}.
\newblock A variational inequality approach to {H}ele--{S}haw flow with a
  moving boundary.
\newblock {\em Proceedings of the Royal Society of Edinburgh: Section A
  Mathematics}, 88(1-2):93--107, 1981.

\bibitem{elliott1986mesa}
CM~Elliott, MA~Herrero, JR~King, and JR~Ockendon.
\newblock The mesa problem: Diffusion patterns for
  ut=∇{\textperiodcentered}(um∇ u) as m→+∞.
\newblock {\em IMA journal of applied mathematics}, 37(2):147--154, 1986.

\bibitem{escher1997classical}
Joachim Escher and Gieri Simonett.
\newblock Classical solutions of multidimensional {H}ele--{S}haw models.
\newblock {\em SIAM Journal on Mathematical Analysis}, 28(5):1028--1047, 1997.

\bibitem{gravner2000internal}
Janko Gravner and Jeremy Quastel.
\newblock Internal dla and the stefan problem.
\newblock {\em Annals of probability}, pages 1528--1562, 2000.

\bibitem{Guillen2015neumann}
Nestor Guillen and Russell~W Schwab.
\newblock Neumann homogenization via integro-differential operators, part 2:
  singular gradient dependence.
\newblock {\em arXiv preprint arXiv:1512.06027}, 2015.

\bibitem{Guillen2014neumann}
Nestor Guillen and Russell~W Schwab.
\newblock Neumann homogenization via integro-differential operators.
\newblock {\em Discrete and Continuous Dynamical Systems, Series A, to appear},
  2016.

\bibitem{jensen}
Robert Jensen.
\newblock Uniqueness criteria for viscosity solutions of fully nonlinear
  elliptic partial differential equations.
\newblock {\em Indiana Univ. Math. J.}, 38(3):629--667, 1989.

\bibitem{MR2203166}
David Jerison and Inwon Kim.
\newblock The one-phase {H}ele--{S}haw problem with singularities.
\newblock {\em J. Geom. Anal.}, 15(4):641--667, 2005.

\bibitem{Kassmann2014Parabolic}
Moritz Kassmann and Russell~W. Schwab.
\newblock Regularity results for nonlocal parabolic equations.
\newblock {\em Riv. Math. Univ. Parma (N.S.)}, 5(1):183--212, 2014.

\bibitem{MR2218549}
Inwon Kim.
\newblock Long time regularity of solutions of the {H}ele--{S}haw problem.
\newblock {\em Nonlinear Anal.}, 64(12):2817--2831, 2006.

\bibitem{kim2003uniqueness}
Inwon~C Kim.
\newblock Uniqueness and existence results on the {H}ele--{S}haw and the stefan
  problems.
\newblock {\em Archive for rational mechanics and analysis}, 168(4):299--328,
  2003.

\bibitem{MR1994745}
Inwon~C. Kim.
\newblock Uniqueness and existence results on the {H}ele--{S}haw and the
  {S}tefan problems.
\newblock {\em Arch. Ration. Mech. Anal.}, 168(4):299--328, 2003.

\bibitem{kim2006regularity}
Inwon~C. Kim.
\newblock Regularity of the free boundary for the one phase {H}ele--{S}haw
  problem.
\newblock {\em J. Differential Equations}, 223(1):161--184, 2006.

\bibitem{kim2009homogenization}
Inwon~C Kim and Antoine Mellet.
\newblock Homogenization of a {H}ele--{S}haw problem in periodic and random
  media.
\newblock {\em Archive for rational mechanics and analysis}, 194(2):507--530,
  2009.

\bibitem{Kinderlehrer1977regularity}
David Kinderlehrer and L~Nirenberg.
\newblock Regularity in free boundary problems.
\newblock {\em Annali della Scuola Normale Superiore di Pisa-Classe di
  Scienze}, 4(2):373--391, 1977.

\bibitem{EJM:2320324}
J.~R. King, A.~A. Lacey, and J.~L. Vazquez.
\newblock Persistence of corners in free boundaries in {H}ele--{S}haw flow.
\newblock {\em European Journal of Applied Mathematics}, 6:455--490, 10 1995.

\bibitem{Koch2015Higher}
Herbert Koch, Arshak Petrosyan, and Wenhui Shi.
\newblock Higher regularity of the free boundary in the elliptic signorini
  problem.
\newblock {\em Nonlinear Analysis: Theory, Methods \& Applications}, 126:3--44,
  2015.

\bibitem{richardson1972hele}
Stanley Richardson.
\newblock {H}ele {S}haw flows with a free boundary produced by the injection of
  fluid into a narrow channel.
\newblock {\em Journal of Fluid Mechanics}, 56(04):609--618, 1972.

\bibitem{saffman1958penetration}
Philip~Geoffrey Saffman and Geoffrey Taylor.
\newblock The penetration of a fluid into a porous medium or {H}ele--{S}haw
  cell containing a more viscous liquid.
\newblock {\em Proceedings of the Royal Society of London. Series A.
  Mathematical and Physical Sciences}, 245(1242):312--329, 1958.

\bibitem{sakai1993regularity}
Makoto Sakai.
\newblock Regularity of boundaries of quadrature domains in two dimensions.
\newblock {\em SIAM journal on mathematical analysis}, 24(2):341--364, 1993.

\bibitem{Schwab2014regularity}
Russell~W Schwab and Luis Silvestre.
\newblock Regularity for parabolic integro-differential equations with very
  irregular kernels.
\newblock {\em arXiv preprint arXiv:1412.3790}, 2014.

\bibitem{serra2014c}
Joaquim Serra.
\newblock ${C}^{\sigma+\alpha}$ regularity for concave nonlocal fully nonlinear
  elliptic equations with rough kernels.
\newblock {\em arXiv preprint arXiv:1405.0930}, 2014.

\bibitem{MR3385173}
Joaquim Serra.
\newblock Regularity for fully nonlinear nonlocal parabolic equations with
  rough kernels.
\newblock {\em Calc. Var. Partial Differential Equations}, 54(1):615--629,
  2015.

\bibitem{Silvestre2007obstacle}
Luis Silvestre.
\newblock Regularity of the obstacle problem for a fractional power of the
  {L}aplace operator.
\newblock {\em Communications on pure and applied mathematics}, 60(1):67--112,
  2007.

\bibitem{MR2737806}
Luis Silvestre.
\newblock On the differentiability of the solution to the {H}amilton-{J}acobi
  equation with critical fractional diffusion.
\newblock {\em Adv. Math.}, 226(2):2020--2039, 2011.

\end{thebibliography}

\end{document}